\documentclass[11pt]{preprint}

\usepackage{times}
\usepackage{latexsym,amssymb, graphicx, amsthm, amsmath, mathrsfs,mathtools,xparse}
\usepackage{mhequ}
\usepackage{cprotect}
\usepackage{hyperref}
\usepackage{breakurl}
\usepackage{trees}
\usepackage{comment}
\usepackage{microtype} 

\newtheorem{theorem}{\bf  {Theorem}}

\newtheorem{lemma}[theorem]{\bf Lemma}

\newtheorem{definition}[theorem]{\bf Definition}

\newtheorem{corollary}[theorem]{\bf Corollary}

\newtheorem{remark}[theorem]{\bf Remark}

\newtheorem{assumption}[theorem]{\bf Assumption}

\numberwithin{theorem}{section}

\definecolor{darkred}{rgb}{0.7,0,0}
\definecolor{darkblue}{rgb}{0,0,0.7}

\def\onorm#1{| #1|_0}
\newcommand{\VERT}{\vert\!\vert\!\vert}
\newcommand{\les}{\lesssim}
\newcommand{\defeq}{\stackrel{\mathclap{\mbox{\tiny def}}}{=}}
\newcommand{\Torus}{\mathbb{T}^3}
\newcommand{\DTorus}{{\mathbb{T}_\varepsilon^3}}
\newcommand{\eBox}[1]{{\square_#1^\varepsilon}}
\newcommand{\alphaeps}{{\alpha + \frac{\kappa}{2}}}

\def\gen{\mathrm{gen}}
\def\CS{\mathcal{S}}
\def\CU{\mathcal{U}}
\def\CK{\mathcal{K}}

\def\CQ{\mathcal{Q}}
\def\CR{\mathcal{R}}
\def\CB{\mathcal{B}}
\def\CC{\mathcal{C}}
\def\CD{\mathcal{D}}

\def\CT{\mathcal{T}}
\def\CF{\mathcal{F}}
\def\CA{\mathcal{A}}
\def\CG{\mathcal{G}}
\def\CM{\mathcal{M}}

\def\1{\mathbf{1}}

\def\ST{\mathscr{T}}
\def\SF{\mathscr{F}}
\def\s{\mathfrak{s}}
\def\FP{\mathfrak{p}}
\def\CP{\mathcal{P}}
\def\CZ{\mathcal{Z}}
\def\spanning{\mathop{\mathrm{span}}} 

\def\rho{\varrho}
\def\STpoly{\mathscr{T}_\text{poly}}
\def\CTpoly{\mathcal{T}_\text{poly}}
\def\CGpoly{\mathcal{G}_\text{poly}}
\def\CFpoly{\mathcal{F}_\text{poly}}
\def\NBar{\bar{\mathbb{N}}}
\let\eps\varepsilon
\def\fT{\mathfrak{T}}
\def\R{\mathbb{R}}
\def\scal#1{\langle#1\rangle}

\title{The $\Phi_3^4$ measure has sub-Gaussian tails}

\author{Martin Hairer and Rhys Steele}

\institute{Imperial College London \\ 
\noindent\email{m.hairer@imperial.ac.uk, r.steele18@imperial.ac.uk}}

\begin{document}
\maketitle

\begin{abstract}
We provide a very simple argument showing that the $\Phi^4_3$ measure does have
quartic exponential tails, as expected from its formal expression. This shows that
the corresponding moment problem is well-posed and provides a simple path to observing non-Gaussianity of the measure.
\end{abstract}


\section{Introduction}

It has been known since the groundbreaking work of Osterwalder and Schrader \cite{OS, OS2} that,
in some ``nice'' settings, the construction of a (bosonic) quantum field theory satisfying the Wightman axioms is equivalent to the construction of a probability measure on the 
space of distributions 
satisfying a number of natural properties. One of the pinnacles of that line of enquiry
was the construction in the seventies of the $\Phi_3^4$ measure
\cite{MR0231601,MR0359612,MR0408581,MR0416337,MR0384003},
which corresponds to the simplest case of an interacting theory in three space-time dimensions. 

Stated superficially, and in the variant derived from the results of the paper \cite{EE78}, the Osterwalder--Schrader axioms require that the Schwinger functions (or $n$-point correlation functions) corresponding to the measure satisfy a regularity assumption, invariance by certain Euclidean transformations, a symmetry axiom, reflection positivity and a clustering property. The last of these assumptions is not mentioned in \cite{EE78} but is present in \cite{OS, OS2}
where it is used independently of the other axioms to obtain the corresponding assumption in Wightman's axioms. (This is then used to derive uniqueness of the vacuum state of the associated quantum field theory.) Whilst the phase-cell expansion approach of Glimm and Jaffe \cite{MR0408581} and subsequent applications of cluster expansion methods by \cite{MR0416337, MS76} was sufficient to verify all of the Osterwalder--Schrader axioms, the level of exposition was reasonably complex and thus attempts to simplify various components were born. Whilst for brevity we cannot point to all such works since the traditional CQFT approach is not the main focus of this paper, we point the reader to \cite{MR723546} for a clean construction via the so-called skeleton inequalities (which nonetheless turns out to be insufficient to verify rotation invariance).

In recent years, a new construction of the $\Phi_3^4$ measure was given by a number of authors
\cite{GH21,MW18b}, mainly relying on stochastic quantisation \cite{ParisiWu} (though we refer the reader also to \cite{BG20b} for a construction via Girsanov's theorem and to \cite{BG20a} for an explicit variational formula). These ``dynamical''
constructions have the advantage of being able to leverage SPDE techniques to obtain very fine 
local properties for the resulting measure. 
It has furthermore been possible to verify all of the Osterwalder--Schrader axioms in the form of \cite{EE78}, except for rotation invariance  and the clustering property (both of which are only relevant when
considering the whole space). The goal of this work is to further leverage the techniques of stochastic quantisation and recent advancements in the field of singular SPDEs to obtain exponential integrability results for the $\Phi_3^4$ measure.

At a heuristic level, the $\Phi_3^4$ measure on the torus $\Torus$ is given up to normalisation by
$$\mu \sim \exp \left( \int -\frac12 |\nabla \Phi(x)|^2 - \frac14 \Phi(x)^4 dx \right)  \prod_{x \in \Torus} d\Phi(x)$$
for $\Phi \in \CS'(\Torus)$.
The above expression is purely a heuristic, both because the product measure appearing is not a well-defined object but also because once that measure is successfully interpreted one ends up in a situation where there is need for renormalisation. 
Nevertheless, this formal expression strongly suggests that, given any test function $\psi$, there
exists $\beta > 0$ such that the function $\Phi \mapsto \exp(\beta \scal{\Phi,\psi}^4)$
is integrable with respect to $\mu$.
The goal of this article is to prove precisely this result,
which is expected to be optimal based on the formal expression for
the measure.

Our proof strongly relies on slight modifications of the a priori bounds obtained in 
\cite{MW18b}, but is otherwise very elementary.
As proposed by Parisi and Wu \cite{ParisiWu}, we interpret $\mu$ as the
invariant measure of the $\Phi_3^4$ equation \cite{Hai14}, which was shown to exist in 
\cite{MW17c} and is unique by \cite{HM18,HS19}. Formally, this equation is given by
\begin{equs}[e:Phi]
	\partial_t \Phi = \Delta \Phi + \infty \Phi - \Phi^3 + \xi\;, \qquad \Phi(0, \cdot) = \Phi_0(\cdot)\;, \tag{$\Phi^4_3$}
\end{equs}
on the torus $\Torus \defeq (\mathbb{R} / \mathbb{Z})^3$ for $t > 0$ where $\Delta$ is the Laplacian on $\Torus$, $\xi$ is a space-time white noise, and $\Phi_0 \in C^\alpha$ for some $\alpha > -\frac 23$ is the initial condition. 
It is known that for sufficiently small coupling constant, the invariant measure 
for \eqref{e:Phi} does indeed coincide with the $\Phi^4_3$ measure as previously
constructed (see \cite{HM15}). One advantage of the recent constructions however is that they 
do not rely on any smallness condition for the coupling constant.

Of course, one must correctly interpret the term $\infty \Phi$ appearing in \eqref{e:Phi}, which corresponds to the need for renormalisation for this equation to be well posed. Indeed, even the solution to the linear part of the equation in spatial dimension $3$ is a Schwartz distribution rather than a function and hence the cubic term is not a priori well-defined. As was shown in \cite{Hai14} (see also \cite{CC13} for an approach via the paracontrolled calculus of \cite{GIP15}), the correct interpretation of a solution to \eqref{e:Phi} is as the limit in probability as $\delta \to 0$ of the solutions to the equations
$$\partial_t \Phi^{(\delta)} = \Delta \Phi^{(\delta)} + (3C_1^{(\delta)} - 9C_2^{(\delta)})\Phi^{(\delta)} - (\Phi^{(\delta)})^3 + \xi_\delta\;, \qquad \Phi(0, \cdot) = \Phi_0(\cdot)\;, \tag{$\Phi^4_3$}$$
where $C_i^{(\delta)}$ are sequences of diverging constants and $\xi_\delta$ is the mollification of $\xi$ at scale $\delta$. Whilst the choice of renormalisation constants depends on the choice of mollifier, the limiting object obtained in this way is independent of the choice of mollifier. We say that this limit is the solution to \eqref{e:Phi}. 

Here we have glossed over the small detail that there is in fact a one-parameter family of solutions obtained in this way, since perturbing the renormalisation constant by a fixed finite quantity does not affect the convergence result (this parameter is the coupling constant mentioned previously). 
Since our results apply equally to any element of this one-parameter family, we ignore this detail from here onwards, considering the choice of coupling to be fixed.

Since the initial development of a solution theory for \eqref{e:Phi}, there have been a number of results establishing various properties for the solution and the associated semigroup. For example, it was established in \cite{HM18} that 
the semigroup $\CP_t$ associated to \eqref{e:Phi} has the strong Feller property. Combining that work with \cite{MW18a}, one 
corollary of the results of \cite{HS19} is that this semigroup is also exponentially ergodic. 
A key ingredient for this proof is a powerful a priori bound that establishes a ``coming down from infinity'' property for \eqref{e:Phi}. This kind of bound was first established via paracontrolled techniques in \cite{MW17c} and later a much shorter argument that is in flavour based on the theory of regularity structures was given in \cite{MW18b,MW18a}. 

The main result of this paper is an exponential integrability result which is significantly stronger than that required by the Osterwalder--Schrader axioms and stronger than those previously available in the literature.

\begin{theorem}\label{t:main result}
		Fix $\kappa > 0$ sufficiently small and let $\mu_M$ be the invariant measure for \eqref{e:Phi} constructed on $C^{-\frac12 - \kappa}(\Torus_M)$ where $\Torus_M$ denotes the torus of length $M$. Let $\psi:\mathbb{R}^3 \to \mathbb{R}$ be a fixed smooth test function with compact support. For $M$ sufficiently large, interpret $\psi$ as a function on $\Torus_M$ in the natural way and define $V_M: C^{-\frac12-\kappa}(\Torus_M) \to \mathbb{R}$ by
		$$V_M(\Phi) = \frac{\beta}{4} \langle \Phi, \psi \rangle^4$$
		for $\beta > 0$ sufficiently small. Then there is a constant $C$ such that for all $M$ sufficiently large 
		$$\int \exp(V_M) \,d\mu_M \leq C.$$
\end{theorem}

Both Gubinelli--Hofmanová \cite{GH21} and Moinat--Weber \cite{MW18b} 
had previously obtained stretched exponential integrability
for any exponent strictly less than~$1$ using SPDE techniques. Whilst this is sufficient to verify 
the regularity axiom in the form of \cite{EE78}, it is insufficient for the form stated in \cite{GJ12} for the purpose of simplifying the exposition there. Theorem~\ref{t:main result} is sufficiently strong to prove that this stronger assumption is satisfied and appears to be the first result in the
literature that yields better than Gaussian tails for the $\Phi^4_3$ measure.
The best bounds obtained using phase cell expansion techniques appear to be slightly worse than Gaussian \cite[Thm~I.3]{MS76}.
See also \cite[Lem.~1.3]{MR0416337} for a proof of exponential integrability. In two dimensions however, 
bounds of the type given in Theorem~\ref{t:main result} were previously obtained by Fröhlich in \cite[Thm~4.8(5)]{MR436831}.

Additionally, beyond being of interest in its own regard for providing what we expect to be optimal integrability for spatial averages of the $\Phi_3^4$-measure, Theorem~\ref{t:main result} is also of interest since it provides sufficient conditions for the moment problem for the $\Phi_3^4$-measure to be well-posed and also provides a new and simple way to observe non-Gaussianity of this measure since a Gaussian measure would not satisfy such an integrability condition (though the latter result was already obtained in the SPDE literature via more involved means; see e.g. \cite[Theorem 5.4]{GH21}).

The key ingredient of our approach is interpreting the integrability statement as corresponding to finiteness of the measure $\exp(V) d\mu$ (suppressing here the dependence on $M$). In the same way that one expects the $\Phi_3^4$ measure to be invariant for \eqref{e:Phi}, one expects $\exp(V) d\mu$ to be invariant for a certain singular SPDE, which we later label \eqref{e:Psi}. Hence, as is usual in the program of stochastic quantisation, we proceed to study this measure via the equation \eqref{e:Psi}.
Without any loss of generality, in what follows we will consider only the case where the test function $\psi$
is such that $\|\psi\|_\infty, \|D\psi\|_\infty \leq 1$.

Finally, we observe that the ideas behind the proof of our main result are not restricted to the particular form of $V$ contained there. In principle, our techniques should be adaptable to obtain exponential integrability of observables that don't exhibit higher than 4th order behaviour and that do not introduce a requirement for additional renormalisation. As a particular instance of this, an appropriate application of the ideas outlined in Section~\ref{s:Main Result} would also yield the following exponential integrability result for Sobolev type norms. We do not include a proof of this theorem since it requires only trivial adaptations of the proof of Theorem~\ref{t:main result}.

\begin{theorem}\label{t: Sobolev bound}
	Given $\alpha > 0$ and $\Phi \in C^{- \alpha}(\Torus)$, we define its homogenous Sobolev norm by $|\Phi|_{-\alpha}^2 \defeq \langle \Phi, \Phi \ast K_{\alpha} \rangle$, where $K_{\alpha}$ is the convolution kernel associated to $\Delta^{- \alpha}$. For fixed $\kappa > 0$, we then define $W: C^{-\frac12-\kappa} \to \mathbb{R}$ by
	$$W(\Phi) = \frac{\beta}{4} |\Phi|_{-\frac12 - \kappa}^4\;.$$
	Then $\exp(W) \in L^1(\mu)$ for $\beta > 0$ sufficiently small.
\end{theorem}

Similarly to above, it is also possible to obtain a statement independent of the torus size (which
is here fixed at $1$), but this then requires restricting the integral of $\scal{\cdot,\cdot}$ to a bounded region.

\begin{remark}
We restrict our statement and exposition to the case of dimension $d=3$, but the case $d=2$ 
is simpler and the exact same argument works. Actually, one should even be able to treat 
the case $d=4-\kappa$ in the sense of \cite{Phi4eps} but that would require some modifications
to our argument.
\end{remark}

\begin{remark}
Let us briefly remark on the link to the results obtained in \cite{GarbanNewman} which suggest even stronger
tail behaviour. Here, we consider the case $d=2$ and recall that ``the'' $\Phi^4_2$ measure is really a two-parameter family
$\{\mu_{c,M}\}_{c \in \R, M \ge 1}$, where $c$ denotes the ``finite part'' of the renormalisation constant and $M$
denotes the finite volume cut-off. Write also $\psi_L(x) = \psi(x/L)$ for $L \ge 1$. 
Our result then shows that for any {\em fixed} $L$, one has 
$\mu_{c,M}(\scal{\Phi, \psi_L} > K) \lesssim \exp(-C_L K^4)$, for some constant $C_L>0$, uniformly over $M$ and
locally uniformly over $c$.

On the other hand, \cite{GarbanNewman} suggests that, for $c$ equal to its critical value $\bar c$, one has
$\mu_{\bar c,M}(\scal{\Phi, \psi_L} > K) \lesssim \exp(-C L^{-2} K^{16})$, provided that $M \gg L$ and 
$L$ is sufficiently large as a function of $K$. (More precisely, one first fixes $x = K L^{-1/8}$ and sends $L \to \infty$,
then considers $x$ large.) For $c \neq \bar c$ however one expects that,
provided that $\psi$ integrates to $0$, one has
$\mu_{c,M}(\scal{\Phi, \psi_L} > K) \lesssim \exp(-C L^{-2} K^{2})$, again for $M \gg L$ and $L$ sufficiently
large. There is of course no contradiction between these bounds since they apply to 
non-overlapping regimes.
\end{remark}

\subsubsection*{Notation and Conventions}

Throughout this article we fix the usual parabolic scaling of $\mathbb{R}^{4} = \mathbb{R}^{1+3}$ so that for a space-time point $z = (t,x)$, $\|z\| = |t|^{\frac12} \vee |x|$ where $|\cdot|$ is the $\ell^\infty$ norm. Additionally, we consider $\mathbb{R}^3$ as being equipped with its usual Euclidean scaling and as corresponding to the `spatial variables' in $\mathbb{R}^4$. 

The scale of regularity of functions in which we will be interested is that of (parabolic) H\"older spaces. For $r > 0$, we let $C^r = C^r(\mathbb{R}^d)$ be the usual space of $r$-H\"older continuous functions. We remark that in the case $r \in \mathbb{N}$, this space consists of $(r - 1)$-times continuously differentiable functions whose $(r - 1)$-th derivative is Lipschitz continuous rather than the smaller space of $r$-times continuously differentiable functions. Further, for $r > 0$ we denote by $\CB_0^r$ the set of $C^r$ functions with support in the parabolic ball of radius $1$ centered at $0$ with $C^r$ norm at most $1$. Throughout the article, one should think of $r$ as a sufficiently large fixed integer. 

For $\alpha < 0$, we let $C^\alpha = C^\alpha(\mathbb{R}^4)$ be the space of Schwartz distributions $\zeta \in \CS'(\mathbb{R}^4)$ that lie in the dual of the space of compactly supported $C^r$ functions for $r > - \lfloor \alpha \rfloor$ such that 
$$\|\zeta\|_\alpha \defeq \sup_{\varphi \in \CB_0^r} \sup_{z \in \mathbb{R}^4} \sup_{\lambda \in (0,1]} \lambda^{-\alpha} |\langle \zeta, \varphi_z^\lambda \rangle | < \infty$$
where $\varphi_z^\lambda(s,y) = \lambda^{-5} \varphi(\lambda^{-2}(s - t), \lambda^{-1}(y-x))$ for $z = (t,x)$. We adopt a similar definition for H\"older spaces of negative regularity over $\mathbb{R}^3$ in which we replace the parabolic scaling with the Euclidean one in the obvious manner. 

We will fix the values $\alpha = - \frac12 - \kappa, \alpha' = \alphaeps$ for $\kappa$ as in Theorem \ref{t:main result}. The important feature of this choice is that all results regarding existence of solutions or convergence of approximations will hold on both $C^\alpha$ and $C^{\alpha'}$, allowing us to at times exploit the compactness of the embedding $C^{\alpha'} \xhookrightarrow{} C^\alpha$.

Finally, for convenience later, we introduce the notation $\NBar \defeq \mathbb{N} \cup \{\infty\}$.

\subsection{Article Structure}
In Section~\ref{s:Main Result}, we gather the statements of results from later sections that are necessary to complete the proof of Theorem~\ref{t:main result} weakened to allow the constant $C$ appearing in the conclusion to depend on the size of the torus. In this same section we then complete said proof. The purpose of first proving this weaker result is to make clearer the key ideas behind our proof. In Section~\ref{s: Extensions}, we give the adaptations necessary to Section~\ref{s:Main Result} to obtain Theorem~\ref{t:main result}. The subsequent sections then contain the technical details of adapting the required results in the literature to our desired setting. In particular, in Section~\ref{s:reg} we introduce elements of the theory of regularity structures \cite{Hai14} and their inhomogeneous models, as introduced in \cite{HM15}. In particular, we show that the equations \eqref{e:TPsi} introduced in Section~\ref{s:Main Result} have a solution theory in this framework that yields global in time solutions. In Section~\ref{s: Cont Bounds}, we will further show that said solutions satisfy a certain a priori bound uniformly in $n$. Finally, in Section~\ref{s:Dreg} we recall details of the discretisation of regularity structures as found in \cite{HM15} (see also \cite{ErHa19}). The main result of this section is the convergence of a family of spatially discrete approximations to the solution of \eqref{e:TPsi}.

\subsection*{Acknowledgements}

{\small
MH gratefully acknowledges support from the Royal Society through a research professorship.
We are grateful to the referees and to Abdelmalek Abdesselam for a number of pointers to the literature 
that were missing in an earlier draft.
}

\section{Proof of Theorem~\ref{t:main result} for Fixed Volume}\label{s:Main Result}

In this section, we will suppress the dependency on the torus length $M$ in our notations. The results obtained will apply for any fixed torus length but in this section we will not obtain a bound that is uniform in $M$.

 As mentioned in the introduction, the main insight of our approach to the proof of Theorem~\ref{t:main result} is to consider the measure $\exp(V)d \mu$ for $V(\Phi) = \frac{\beta}{4} \langle \Phi, \psi \rangle^4$. By analogy to the classical setting of a one-dimensional stochastic gradient flow, where identification of invariant measures is reduced to a simple calculation, if we were to ignore the effect of singularities involved then we would expect $\exp(V)d \mu$ to be the invariant measure for the equation
\begin{equs}[e:Psi]
	\partial_t \Psi = \Delta \Psi + \infty \Psi - \Psi^3 + \beta \langle \Psi, \psi \rangle^3 \psi +  \xi\;, \qquad \Psi(0, \cdot) = \Psi_0(\cdot)\;,
	\tag{$\Psi^4_3$}
\end{equs}
where $\langle \cdot, \psi \rangle$ refers to testing in space only.

It is not immediate from the constructions of \cite{Hai14} that this equation has a solution theory provided by the framework of regularity structures since the additional nonlinearity appearing is a non-local one. In Section~\ref{s:reg}, we show that this additional nonlinearity poses no serious trouble in building such a solution theory for \eqref{e:Psi} using the techniques of regularity structures. Our preferred approach in this section is that of inhomogeneous models as first presented in \cite{HM15}. This approach is advantageous both because in the additional nonlinearity time plays a distinguished function-like role and because later we will want to discretise in space (but not time) again giving the time variable a distinguished role.

Combining Theorem~\ref{t: Fixed Point} and Remarks~\ref{r: classical solutions},~\ref{r:NoBlowUp} yields global in time solutions to equation \eqref{e:Psi}. Additionally, these solutions are given as limits in probability as $\delta \to 0$ of the pathwise constructed solution to the random PDE
\begin{align}\label{e:Psi renorm}
(\partial_t - \Delta)u = -u^3 +(3C_1^\delta - 9C_2^\delta)u + \beta \psi \langle u, \psi \rangle^3  + \xi_\delta. 
\end{align}
where $C_i^\delta$ are renormalisation constants that diverge as $\delta \to 0$ and $\xi$ is the mollification of space-time white noise at scale $\delta$. 

This allows us to leverage the techniques used in \cite{MW18a} to prove a priori bounds on the solution to \eqref{e:Psi} that are uniform in the initial condition. If the identification of $\exp(V) d\mu$ as an invariant measure for \eqref{e:Psi} were more than a heuristic, we could conclude the proof of Theorem~\ref{t:main result} by considering the solution to the equation started from the invariant distribution. Unfortunately, this is not the case and as a result we will need a priori bounds in a more general setting than stated above; hence we defer the statement of such bounds until we are in this setting.

To overcome this issue, we proceed in two stages of approximation. First we truncate the additional nonlinearity appearing in equation \eqref{e:Psi}. For $n \in \mathbb{N}$, let $F_n:\mathbb{R} \to \mathbb{R}$ be a smooth function such that
$$F_n(x) = \begin{cases}
\frac14x^4, \qquad &|x| \leq n 
\\
\frac14n^4 + 1, \qquad &|x| \geq n + 1
\end{cases}
$$ $|F_n(x)| \in [n^4, n^4 + 1]$ for $|x| \in [n,n+1]$ and $|F_n'(x)| \leq n^3$ for all $x \in \mathbb{R}$. We then consider the equations (indexed by $n \in \mathbb{N}$)
\begin{equs}[e:TPsi]
	\partial_t \Psi^{(n)} = \Delta \Psi^{(n)} + \infty \Psi^{(n)} - (\Psi^{(n)})^3 + \beta F_n'(\langle \Psi^{(n)}, \psi \rangle) \psi +  \xi\;
	\quad \tag{$\Psi^{4,n}_3$}
\end{equs}
with the same initial condition $\Psi^{(n)}(0, \cdot) = \Psi_0(\cdot)$.

We additionally let $F_\infty(x) = \frac14 x^4$ so that equation \eqref{e:Psi} corresponds to \eqref{e:TPsi} in the case $n = \infty$.

Equations \eqref{e:TPsi} are again formulated in the framework of regularity structures in Section~\ref{s:reg}. In Section~\ref{s: Cont Bounds}, we adapt the techniques of \cite{MW18a} to prove the following a priori bound which is now uniform in both the initial condition, choice of $\psi$ and in $n$. 

\begin{theorem}\label{t:A Priori Bound}
	Fix a function $\psi \in \CB_0^r$ and fix also $\beta > 0$ to be sufficiently small. If $\Psi^{(n)}$ is the solution of \eqref{e:TPsi} for this $\psi$, then for all $R \in (0,1)$ one has that
	$$\sup_{ t \in (R^2,1)} \|\Psi^{(n)}(t)\|_{C^\alpha} \leq 1 \vee C \max\{ R^{-1}, \|\<1>\|_{\infty, -\frac12 -; (0,1)}, [\tau]_{|\tau|}^\frac{1}{n_\tau (\frac12 -)}; \tau \in \fT\}$$
	where $\fT$ is a collection of trees constructed from the driving noise. Both this collection of trees and the various norms appearing on the right hand side are defined in Section~\ref{s: Cont Bounds}
\end{theorem}

In a next stage of approximation, in Section~\ref{s:Dreg}, for each $n$ we discretise space to obtain a system of SDEs, labelled \eqref{e:DPsiRenorm}, approximating \eqref{e:TPsi}. Simultaneously we consider the equivalent discrete approximations of \eqref{e:Phi} as considered in e.g. \cite{HM15, GH21}. The main result of this section is the convergence of the discrete approximations to the solution of the corresponding continuum equation as the grid scale is sent to $0$.

The purpose of these two stages of approximation is as follows. Denoting the invariant measure of the discrete approximations to \eqref{e:TPsi} at grid-scale $\varepsilon$ by $\nu_\varepsilon^n$ and that of the discrete approximation to \eqref{e:Phi} by $\mu_\varepsilon$ we have that 
$$d\nu_\varepsilon^n = \CZ_{n, \varepsilon}^{-1}\exp\big(\beta F_n(\langle \iota^\varepsilon \cdot, \psi \rangle)\big)\, d \mu_\varepsilon$$
where $\iota^\varepsilon: \mathbb{R}^\DTorus \to C^\alpha$ interprets a function on the discretised torus $\DTorus \defeq \varepsilon \mathbb{Z}^3 \cap \mathbb{T}$ as a distribution via piecewise constant extension by setting
$$\langle \iota^\varepsilon F, \varphi \rangle \defeq \sum_{y \in \DTorus} \int_\eBox{y} F(y) \varphi(z) dz$$
where $\eBox{y} \defeq \{z \in \Torus: \|z-y\|_\infty \leq \frac{\varepsilon}{2}\}$. In particular, we can exploit the boundedness of this density to identify that $d \nu^n = \CZ_n^{-1}\exp(\beta F_n(\langle \cdot, \psi \rangle)) d\mu$. This knowledge, combined with the a priori bounds of Section~\ref{s: Cont Bounds} allows us to conclude the proof of Theorem~\ref{t:main result} for fixed volume. 

The rest of this section will complete the details missing from the remarks in the above paragraph.

\begin{theorem}\label{t:Weak Convergence}
	The measures $\iota_*^\varepsilon \mu_\varepsilon$ converge weakly on $C^\alpha$ to $\mu$ as $\varepsilon \to 0$ along the dyadics. The same convergence holds for $\iota_*^\varepsilon \nu_\varepsilon^n \to \nu^n$.
\end{theorem}
\begin{proof}
	We begin with the case of $\iota_*^\varepsilon \mu_\varepsilon$. From the results of \cite{GH21}, the family $\iota_*^\varepsilon \mu_\varepsilon$ is tight (they apply their result to the measures on expanding tori, however their bounds are all uniform in the length of the torus considered). Hence, it suffices to show that $\mu$ is the unique limit point of the sequence $\iota_*^\varepsilon \mu_\varepsilon$. Then by \cite[Chapter 4, Theorem 4.5]{EK09} it suffices to show that if $\tilde{\mu}$ is a weak limit point of $\iota_*^\varepsilon \mu_\varepsilon$ then $\tilde{\mu}(f) = \mu(f)$ for all bounded Lipschitz functions $f$ on $C^\alpha$ with Lipschitz constant at most $1$ since this set of functions separates points in $C^\alpha$.
	
	Fix such a Lipschitz function $f$ on $C^\alpha$. We will show that $\iota_*^\varepsilon \mu_\varepsilon(f) \to \mu(f)$ which is certainly sufficient for our goal.
	
	By exploiting invariance, we begin with the simple bound
	\begin{equ}\label{e:Initial weak bound}
		|\mu(f) - \iota_*^\varepsilon \mu_\varepsilon(f)| \leq |\mu(f) - \CP_t \iota_*^\varepsilon \mu_\varepsilon(f)| + |\CP_t \iota_*^\varepsilon \mu_\varepsilon(f) - \iota_*^\varepsilon \CP_t^\varepsilon \mu_\varepsilon(f)|
	\end{equ}
	where $\CP_t, \CP_t^\varepsilon$ are the semigroups associated with \eqref{e:Phi} and \eqref{e:DPhiRenorm} respectively.
	
	To control the first term in \eqref{e:Initial weak bound}, we note that the proof of \cite[Corollary~1.9]{HS19} shows that $\CP_t$ satisfies the hypotheses of Harris' Theorem (see e.g.~\cite[Theorem~1.2]{Harris}) and in particular, one even has that $\|\mu - \CP_t \iota_*^\varepsilon \mu_\varepsilon\|_{\text{TV}} \to 0$ at exponential rate as $t \to \infty$.
	
	Therefore for fixed $\eta > 0$, we may fix $t$ sufficiently large such that 
	$$Q_1\defeq|\mu(f) - \CP_t \iota_*^\varepsilon \mu_\varepsilon(f)| < \frac{\eta}{4}$$
	uniformly in $\varepsilon$. 
	
	We now turn to controlling the second term on the right hand side of \eqref{e:Initial weak bound} for this fixed value of $t$. We will prove bounds that would be strong enough to control this term in Wasserstein-1 distance. For this, we write
	$$\CP_t \iota_*^\varepsilon \mu_\varepsilon(f) - \iota_*^\varepsilon \CP_t^\varepsilon \mu_\varepsilon(f) = \int_{C^\alpha} \mathbb{E}[f(\Phi_\phi(t)) - f(\iota^\varepsilon \Phi_{\FP_\varepsilon \phi}^\varepsilon(t))]  \iota_*^\varepsilon \mu_\varepsilon(d\phi)$$
	where $\Phi_\phi(t), \Phi_{\FP_\varepsilon \phi}^\varepsilon(t)$) are the solutions to \eqref{e:Phi}, \eqref{e:DPhiRenorm} started from $\phi$ and $\FP_\varepsilon \phi$) respectively and we have used the fact that $\FP_\varepsilon \defeq \langle \cdot, \varepsilon^{-3} 1_{\{\|\cdot - x\|_\infty \leq \frac{\varepsilon}{2}\}} \rangle$ is a left inverse to $\iota^\varepsilon$ to identify the appropriate initial condition for the discrete dynamic.
	
	By tightness of $\{\iota_*^\varepsilon \mu_\varepsilon: \varepsilon = 2^{-n}, n \geq 1\}$ and boundedness of $f$, there exists a compact set $\CK_\eta$ such that
	$$Q_2 \defeq \int_{\CK_\eta^c} \mathbb{E}[|f(\Phi_\phi(t)) - f(\iota^\varepsilon \Phi_{\FP_\varepsilon \phi}^\varepsilon(t))|]  \iota_*^\varepsilon \mu_\varepsilon(d\phi) < \frac{\eta}{4}.$$
	It remains to consider the integral over $\CK_\eta$. For this, the crucial remark is that it follows from Theorem~\ref{t:Convergence} that $\|\Phi_\phi(t) - \iota^\varepsilon \Phi_{\FP_\varepsilon\phi}^\varepsilon(t)\|_{C^\alpha} \to 0$ as $\varepsilon \to 0$ in probability uniformly over $\phi \in \CK_\eta$. 
	
	As a result, there exists an $\varepsilon_0 \in (0,1)$ such that $ \varepsilon<\varepsilon_0$ implies that $$Q_3 \defeq\mathbb{P}\left(\exists \phi \in \CK_\eta \text{ such that }\|\Phi_\phi(t) - \iota^\varepsilon \Phi_{\FP_\varepsilon\phi}^\varepsilon(t)\|_{C^\alpha} \geq \frac{\eta}{4}\right) \leq \frac{\eta}{8\|f\|_\infty}.$$
	
	Hence, since $f$ is Lipschitz continuous with Lipschitz constant at most $1$, for $\varepsilon < \varepsilon_0$ one has the estimate
	\begin{align*}
	\int_{\CK_\eta} \mathbb{E}[|f(\Phi_\phi(t)) - f(\iota^\varepsilon \Phi_{\FP_\varepsilon \phi}^\varepsilon(t))|]  \iota_*^\varepsilon \mu_\varepsilon(d\phi) \leq 2Q_3 \|f\|_\infty + \frac{\eta}{4} \leq \frac{\eta}{2}.
	\end{align*}
	Combining these estimates, we see that
	\begin{equation*}
	|\CP_t \iota_*^\varepsilon \mu_\varepsilon(f) - \iota_*^\varepsilon \CP_t^\varepsilon \mu_\varepsilon(f)| \leq \frac{3\eta}{4}\;.
	\end{equation*}
	Substituting this bound, along with the bound on $Q_1$ into \eqref{e:Initial weak bound} yields that for $\varepsilon< \varepsilon_0$
	$$|\mu(f) - \iota_*^\varepsilon \mu_\varepsilon(f)| < \eta\;.$$
	which completes the proof for $\iota_*^\varepsilon \mu_\varepsilon$.
	
	It remains to consider $\iota_*^\varepsilon \nu_\varepsilon^n$. The proof will be the same once we obtain tightness and the hypotheses of Harris' theorem. 
	For tightness (at fixed $n$), we first note that since \eqref{e:DPsiRenorm} is nothing but a system of SDEs, a simple calculation using the generator of this system shows that $$d\nu_\varepsilon^n = \CZ_{n, \varepsilon}^{-1} \exp(\beta F_n(\langle \iota^\varepsilon \cdot, \psi \rangle)) d\mu_\varepsilon$$ 
	where $\CZ_{n, \varepsilon}$ is a normalising factor. Hence the desired tightness follows immediately from the fact that $|F_n|$ is bounded and tightness of $\{\iota_*^\varepsilon \mu_\varepsilon: \varepsilon = 2^{-n}, n \geq 1\}$.
	
	It remains to verify the bounds of Harris' theorem. Here we will only point out the adaptations needed to \cite[Corollary 1.9]{HS19}. The results in Section~\ref{s: Cont Bounds} give the required `coming down from infinity property' so that all that remains is to see that $\nu^n$ has full support. This follows from full support of $\mu$ \cite[Theorem 1.8]{HS19} via Girsanov transformation. Indeed, from the solution theory of Section~\ref{s:reg}, if $\Psi$ is the solution of \eqref{e:TPsi} then one can consider the function $H_n \defeq F_n'(\langle \Psi, \psi \rangle) \psi$. $H_n$ is bounded, smooth in space and $\eta$-H\"older continuous in time for $\eta > 0$ sufficiently small. The solution to \eqref{e:TPsi} then coincides with the solution to the equation
	$$\partial_t \Psi^{(n)} = \Delta \Psi^{(n)} + \infty \Psi^{(n)} - (\Psi^{(n)})^3 + \beta H_n +  \xi.$$
	In particular, if $\xi$ is a $\mathbb{P}$-space-time white noise then there exists an equivalent measure $\mathbb{Q}$ such that $H_n + \xi$ is a space-time white noise \cite{Al98}. As one would expect, the results of \cite[Sections 4 and 5.1]{HM18} then verify that if $\Phi$ is the $\mathbb{P}$-solution for \eqref{e:Phi} then under $\mathbb{Q}$, $\Phi$ solves the above equation and hence also \eqref{e:TPsi}. Since $\mathbb{P}$ and $\mathbb{Q}$ are equivalent measures, this gives the desired result.
\end{proof}
\begin{remark}
	Whilst in the proof of Theorem~\ref{t:Weak Convergence}, we control the first term on the right hand side of \eqref{e:Initial weak bound} in total variation distance and the second term in Wasserstein-1 distance, we do not conclude a stronger form of convergence than weak convergence since on $C^\alpha$, neither of total variation convergence and Wasserstein-1 convergence implies the other.
\end{remark}
\begin{remark}
	One may hope to remove the dependence on the results of \cite{GH21} in the proof of Theorem \ref{t:Weak Convergence} by adapting the bounds of \cite{MW18a} to the discrete setting (uniformly in grid scale) and using the fact that these bounds are uniform in the initial condition to conclude tightness of the invariant measures. Unfortunately, in the proof of their boundary value free Schauder estimate, \cite{MW18a} rely on precise formulae for Taylor remainders that are not available in the discrete setting. This is almost certainly a purely technical barrier that could be overcome to make the argument given in this paper independent of the framework of paracontrolled calculus leveraged in \cite{GH21}.
\end{remark}
\begin{corollary}\label{c: density}
	The measure $\nu^n$ has density $\CZ_n^{-1} \exp(\beta F_n(\langle \cdot, \psi \rangle))$ with respect to $\mu$.
\end{corollary}
\begin{proof}
	As noted in the previous proof, $$d\nu_\varepsilon^n = \CZ_{n, \varepsilon}^{-1} \exp(\beta F_n(\langle \iota^\varepsilon \cdot, \psi \rangle)) d\mu_\varepsilon\;,$$
	and	hence
	$$ d\iota_*^\varepsilon \nu_\varepsilon^n = \CZ_{n,\varepsilon}^{-1} \exp(\beta F_n(\langle  \cdot, \psi \rangle)) d\iota_*^\varepsilon \mu_\varepsilon\;.$$
	
	By the weak convergence of $\iota_*^\varepsilon \mu_\varepsilon$, $\CZ_{n, \varepsilon} \to \CZ_n$ as $\varepsilon \to 0$.
	Therefore, by the same weak convergence, the integrals against $\CZ_n^{-1} \exp(\beta F_n(\langle \cdot, \psi \rangle)) d \mu$ and $d \nu^n$ agree on continuous bounded functions and hence the two measures are equal.
\end{proof}

\begin{lemma}\label{l:tightness}
	For $\beta$ small enough, the family of measures $\{\nu^n: n \in \mathbb{N}\}$ is tight.
\end{lemma}
\begin{proof}
	Since the bounds of Section~\ref{s: Cont Bounds} (see Theorem~\ref{t:continuum bound}) are uniform in the initial condition and in $n$, for any $\gamma > 0$ there is a $K$ such that
	$$\sup_{n} \mathbb{P}(\|\Psi^{(n)}(1)\|_{C^{\alpha'}} \geq K) \leq \gamma\;,$$
	where $\Psi^{(n)}$ is the solution to \eqref{e:TPsi} started from $\nu^n$. Since $\nu^n$ is invariant for these dynamics, this is nothing but
	$$\inf_{n} \nu^n(\CB_{\alpha'}(K)) \geq 1 - \gamma\;,$$
	where $\CB_{\alpha'}(K)$ is the closed ball of radius $K$ in $C^{\alpha'}$. Since the embedding $C^{\alpha'} \to C^\alpha$ is compact, this implies the desired result.
\end{proof}
Whilst in this case it is possible to establish tightness of the family $\nu^n$, such a strong condition is not actually necessary to apply our techniques. The above proof has as a corollary the following result, which is enough for our purposes.
\begin{corollary}\label{c:ball of positive mass}
	There exists $K > 0$ such that $\inf_{n} \nu^n(\CB_\alpha(K)) \geq \frac12$.
\end{corollary}
\begin{remark} \label{r: Blow-Up of Bounds}
	Until Lemma \ref{l:tightness}, the results stated in this section could have been formulated uniformly in the size of the torus taken as the spatial domain. However, implicit in the proof of this Lemma is a dependency on the size of the torus since the constant $K$ depends on the various norms of trees appearing in Theorem \ref{t:A Priori Bound} which explode as the spatial domain is expanded to the whole of $\mathbb{R}^3$.
\end{remark}
\begin{proof}[Proof of Theorem~\ref{t:main result} for Fixed Volume]
	By Corollary~\ref{c:ball of positive mass}, there exists $K > 0$ such that by writing $\CK= \CB_\alpha(K)$, one has that $\mu(\CK) \geq \frac12$ and
	$$\inf_n \nu^n(\CK) = \inf_n \CZ_n^{-1} \int_\CK \exp(\beta F_n(\langle \cdot, \psi \rangle)) d \mu \geq \frac12\;.$$
	Since $\langle \Phi, \psi \rangle \leq \|\Phi\|_{C^\alpha}$, $\exp(\beta F_n(\langle \cdot, \psi \rangle)) \leq C \defeq \exp(\beta K^4)$ on $\CK$. Hence $\frac12 \leq \CZ_n^{-1} C \mu(\CK) $ and so
	$$\CZ_n \leq 2C.$$
	We have that $\CZ_n = \int_{C^\alpha} \exp(\beta F_n(\langle \cdot, \psi \rangle)) d\mu$ and so the result follows by Fatou's lemma.
\end{proof}

\section{Volume Independent Bounds}\label{s: Extensions}

In this section, we give the necessary adaptations to the main techniques of this paper to obtain the result stated in Theorem \ref{t:main result} uniformly in the size of the torus. Since the ideas are almost the same as those given in Section~ \ref{s:Main Result}, we mainly seek to highlight the points in the argument at which one must make adaptations to deal with the particular details in the statements of these results. 

As was mentioned in Remark \ref{r: Blow-Up of Bounds}, the single point at which the argument of Section \ref{s:Main Result} is not uniform in the size of the torus is in the choice of constant $K$ in the proof of Lemma \ref{l:tightness}. The necessity of enlarging $K$ as the size of torus considered grows comes from the fact that the H\"older-type seminorms of the various trees in the statement of Theorem~\ref{t:continuum bound} grow with the size of the spatial domain. 

Our strategy to overcome this issue is based on the observation that $V(\Phi)$ depends on $\Phi$ only via its behaviour on the support of $\psi$. Hence we are able to adapt the proof of Lemma~\ref{l:tightness} to include a localisation in space and overcome the dependency on torus length.

\begin{proof}[Proof of Theorem~\ref{t:main result}] 
	Throughout the proof we fix $M$ sufficiently large such that $\operatorname{supp} \psi \subseteq [-M,M]^3$ so that all functions can naturally be interpreted as functions on the torus $\Torus_M$.
	
	We denote by $\mu_M$ and $\nu_M^n$ the invariant measures for the $\Phi_3^4$ and $\Psi_3^{4,n}$ equations on $\Torus_M$ respectively. Exactly as in the proof of Corollary~\ref{c: density}, we obtain that
	$$d \nu_M^n = \CZ_{n,M}^{-1} \exp(\beta F_n(\langle \cdot, \psi \rangle)) d \mu_M.$$
	
	Let $\CK_\psi$ denote the 1-fattening of the support of $\psi$ and define the measures $\tilde{\nu}_M^n = \CM_*\nu_M^n$ and $\tilde{\mu}_M = \CM_* \nu_M$ to be the pushforward measures by the operator $\CM$ which acts on a distribution via multiplication with $1_{\CK_\psi}$. From the above, we then have that
	$$d \tilde{\nu}_M^n = \CZ_{n,M}^{-1} \exp(\beta F_n(\langle \cdot, \psi \rangle)) d \tilde{\mu}_M.$$
	
	We observe that
	\begin{equs}
		\int_{C^\alpha} \exp(V_M(\Phi)) \mu_M(d\Phi) &= \int_{C^\alpha} \exp(V_M(\CM \Phi)) \mu_M(d\Phi)
		\\ &= \int_{C^\alpha} \exp(V_M(\Phi)) \tilde{\mu}_M(d\Phi)
	\end{equs}
	so that it suffices to consider the latter integral. 
	
	This can now be achieved via the same ideas as in Section~\ref{s:Main Result} provided we are careful to be explicit about the values of constants. Indeed, one has that
	\begin{equation} \label{e: inf vol invariance equation}
		\inf_{n,M} \tilde{\nu}_M^n(\CB_\alpha(K)) = \inf_{n,M} \mathbb{P}(\|\Psi_M{(n)}(\gamma) 1_{\CK_\psi}\|_{C^\alpha} \leq K)
	\end{equation}
	where $\Psi_M^{(n)}$ is the solution of \eqref{e:TPsi} and $\gamma \in (0,1)$.
	
	Theorem~\ref{t: localised continuum bound} then implies that one can choose $\gamma$ sufficiently small (independently of $n$ and $M$) such that for $K$ sufficiently large one has that the latter quantity in \eqref{e: inf vol invariance equation} is greater than $\frac{1}{2}$.
	
	The remainder of the proof then follows as in the proof of Theorem~\ref{t:main result} for a fixed volume given in Section~\ref{s:Main Result}, noting that the constant $C$ appearing is at this point independent of $M$ since the various seminorms of the trees appearing are localised to a domain that is chosen independently of $M$.
\end{proof}

\section{Regularity Structures and Inhomogeneous Models}\label{s:reg}

In this section we recall the definition of a regularity structure and the framework of inhomogeneous models as developed in \cite{HM15}. The significant difference to the setting of \cite{Hai14} is that in the case of inhomogeneous models, the time variable plays a distinguished role and many objects built in the theory are as a result genuine functions in time. This set-up is convenient for establishing a solution theory for \eqref{e:Psi} since the additional nonlinearity requires the ability to test the solution at a fixed time against some test function in space.

\begin{definition}
	A tuple $\ST = (\CT, \CG)$ is a {\it regularity structure} if:
	\begin{itemize}
		\item $\CT$ is a graded vector space $\CT = \bigoplus_{\alpha \in \CA} \CT_\alpha$, where each $\CT_\alpha$ is a Banach space and $\CA \subset \mathbb{R}$ is a locally finite set. $\CT$ is called the {\it model space} of $\ST$.
		
		\item $\CG$ is a group of linear transformations of $\CT$, such that for every $\Gamma~\in~\CG$, every $\alpha \in \CA$ and every $\Abtau \in \CT_\alpha$ one has
		$\Gamma \Abtau - \Abtau \in \CT_{< \alpha}$, with $\CT_{< \alpha} \defeq \bigoplus_{\beta < \alpha} \CT_\beta$. $\CG$ is called the {\it structure group} of $\ST$.
	\end{itemize}
\end{definition}

\begin{remark}
	We have adopted the convention that elements of regularity structures are coloured blue. This will lend clarity since we will later use a graphical notion in two (similar) ways which will be distinguished by colour. The one exception to this colouring convention is that functions $H: [0,T) \times \mathbb{R}^d \to \CT$ won't be coloured since it is always clear in which space they are valued.
\end{remark}

In our setting, we will always work with regularity structures such that each $\CT_\alpha$ is finite-dimensional and $\CA$ is finite. In particular, there is no ambiguity in the choice of topology.

\begin{assumption}
	Throughout this article, we assume that for a fixed $r > 0$ all regularity structures $\ST= (\CT, \CG)$ contain the structure $\STpoly$ of polynomials of scaled degree at most $r$ introduced in \cite[Remark 2.2]{HM15} in the sense that $\CTpoly \subseteq \CT$ and the restriction of the action of $\CG$ to $\CTpoly$ coincides with that of the group $\CGpoly\simeq (\mathbb{R}^{d+1},+)$
	via a group morphism $\CG \to \CGpoly$.
\end{assumption}

Thus far, the setting described corresponds to the setting of \cite{Hai14} up to the fact that we have insisted on truncating our structures at a fixed maximal homogeneity. However, in what immediately follows we depart from the original definitions used there and instead recall the notion of an inhomogeneous model as in \cite[Definition 2.4]{HM15}.

\begin{definition}\label{d:Model}
	For a regularity structure $\ST= (\CT, \CG)$, an inhomogeneous model is a tuple $((\Pi_x^t)_{(t,x) \in \mathbb{R}^{1+d}}, (\Gamma^t)_{t \in \mathbb{R}}, (\Sigma_x)_{x \in \mathbb{R}^d})$ where
	\begin{itemize}
		\item For $t \in \mathbb{R}$, $\Gamma^t : \mathbb{R}^{d} \times \mathbb{R}^{d} \to \CG$, satisfies the algebraic relations 
		\begin{equ}[e:GammaDef]
			\Gamma^t_{x x}=1\;, \qquad \Gamma^t_{x y} \Gamma^t_{y z} = \Gamma^t_{x z}\;,
		\end{equ}
		for any $x, y, z \in \mathbb{R}^{d}$.
		\item For $x \in \mathbb{R}^d$, $\Sigma_x : \mathbb{R} \times \mathbb{R} \to \CG$ satisfies the algebraic relations
		\begin{equ}[e:SigmaDef]
			\Sigma^{t t}_{x}=1\;, \qquad \Sigma^{s r}_{x} \Sigma^{r t}_{x} = \Sigma^{s t}_{x}\;, \qquad \Sigma^{s t}_{x} \Gamma^{t}_{x y} = \Gamma^{s}_{x y} \Sigma^{s t}_{y}\;,
		\end{equ}
		for any $s,r,t \in \mathbb{R}$ and $y \in \mathbb{R}^d$.
		\item For any $(t,x) \in \mathbb{R}^{1+d}$, $\Pi^t_x: \CT \to \mathcal{S}'(\mathbb{R}^{d})$ satisfies the algebraic relation
		\begin{equ}[e:PiDef]
			\Pi^t_{y} = \Pi^t_x \Gamma^t_{x y} 
		\end{equ}
		for all $y \in \mathbb{R}^{d}$.
	\end{itemize}
	
	Additionally, we impose that the actions of $\Gamma^t_{xy}$ and $\Sigma_x^{st}$ on $\CTpoly$ are given by translation by $(0,y-x)$ and $(t-s,0)$ respectively, and that the maps $\Pi$ on $\CTpoly \subseteq \CT$ are given by
	\begin{equs}
		\bigl(\Pi_x^t \absymbol{X^{(0, \bar k)}}\bigr)(y) = (y-x)^{\bar k}, \quad \bigl(\Pi_x^t \absymbol{X^{(k_0, \bar k)}}\bigr)(y) = 0
	\end{equs}
	for $k_0 > 0$. Finally, for any $\gamma > 0$ and every $T > 0$, we assume that there is a constant $C$ for which the analytic bounds
	\minilab{Model}
	\begin{equs}\label{e:PiGammaBound}
		| \langle \Pi^t_{x} \Abtau, \varphi_{x}^\lambda \rangle| \leq C \| \Abtau \| \lambda^{l} &\;, \qquad \| \Gamma^t_{x y} \Abtau \|_{m} \leq C \| \Abtau \| | x-y|^{l - m}\;,\\
		\| \Sigma^{s t}_{x} \Abtau \|_{m} &\leq C \| \Abtau \| |t - s|^{(l - m)/\s_0}\;,\label{e:SigmaBound}
	\end{equs}
	hold uniformly over all $\Abtau \in \CT_l$, with $l \in \CA$ and $l < \gamma$, all $m \in \CA$ such that $m < l$, all $\lambda \in (0,1]$, all $\varphi \in \CB^r_0(\mathbb{R}^d)$ with $r > -\lfloor\min \CA\rfloor$, and all $t, s \in [-T, T]$ and $x, y \in \mathbb{R}^d$ such that $|t - s| \leq 1$ and $|x-y| \leq 1$.
	
	If additionally, for $\eta > 0$ the bound
	\begin{equs}
		\label{e:PiTimeBound}
		| \langle \bigl(\Pi^t_{x} - \Pi^{s}_{x}\bigr) \Abtau, \varphi_{x}^\lambda \rangle| \leq C \| \Abtau \| |t-s|^{\eta/\s_0} \lambda^{l - \eta}\;,
	\end{equs}
	holds for all $\Abtau \in \CT_l$ and the other parameters as before then we say that $\Pi$ has time regularity $\gamma > 0$.
\end{definition}

As is usual, the collection of maps $(\Pi, \Sigma, \Gamma)$ as above that satisfy the analytic constraints but not necessarily the algebraic constraints is a linear space. For any fixed $T > 0$, this space comes equipped with a norm
\begin{equ}
	\VERT Z \VERT_{\gamma; T} \defeq \| \Pi \|_{\gamma; T} + \| \Gamma \|_{\gamma; T} + \| \Sigma \|_{\gamma; T}\;,
\end{equ}
where $Z = (\Pi, \Gamma, \Sigma)$ and $\|\Pi\|_{\gamma; T}, \|\Gamma\|_{\gamma; T}$ and $\|\Sigma\|_{\gamma; T}$ are the smallest constants $C$ such that the analytic bounds in \eqref{e:PiGammaBound} and \eqref{e:SigmaBound} hold for the relevant object. 

In particular, despite not being a linear subspace of this larger space, the space of inhomogeneous models inherits a ``distance'' that is given for a pair of models $Z, \bar{Z}$ by 
\begin{equ}[e:ModelsDist]
	\VERT Z; \bar{Z} \VERT_{\gamma; T} \defeq \| \Pi - \bar{\Pi} \|_{\gamma; T} + \| \Gamma - \bar{\Gamma} \|_{\gamma; T} + \| \Sigma - \bar{\Sigma} \|_{\gamma; T}\;.
\end{equ}
Additionally, if $\Pi$ has time regularity $\eta > 0$ then we can account for this by defining $\|\Pi\|_{\eta, \gamma;T} \defeq \|\Pi\|_{\gamma;T} + C'$ where $C'$ is the smallest constant such that the bound \eqref{e:PiTimeBound} holds. We then define $\VERT Z \VERT_{\eta, \gamma; T}$ and $\VERT Z, \bar{Z}\VERT_{\eta, \gamma; T}$ analogously to the definitions above, replacing all instances of $\|\Pi\|_{\gamma;T}$ with $\|\Pi\|_{\eta, \gamma; T}$.


\subsection{Inhomogeneous Modelled Distributions}

In the following definition, we consider a fixed regularity structure $\ST = (\CT, \CG)$ with inhomogeneous model $Z=(\Pi, \Gamma, \Sigma)$, parameters $\gamma, \eta \in \mathbb{R}$, a time $T > 0$, and 
$H : (0, T] \times \mathbb{R}^{d} \to \CT_{<\gamma}$. We define the quantity
\begin{equs}[e:ModelledDistributionNormSpace]
	\Vert H \Vert_{\gamma, \eta; T} \defeq \sup_{t \in (0,T]} &\sup_{x \in \mathbb{R}^d} \sup_{l < \gamma} \onorm{t}^{(l - \eta) \vee 0} \| H_t(x) \|_l\\
	&+ \sup_{t \in (0,T]} \sup_{\substack{x \neq y \in \mathbb{R}^d \\ | x - y | \leq 1}} \sup_{l < \gamma} \frac{\| H_t(x) - \Gamma^{t}_{x y} H_t(y) \|_l}{\onorm{t}^{\eta - \gamma} | x - y |^{\gamma - l}}\;,
\end{equs}
where $l \in \CA$ in the third supremum and $\onorm{t} \defeq |t|^{\frac12} \wedge 1$. This quantity is a partial analogue of the homogeneous $\CD^{\gamma, \eta}$ norm introduced in \cite{Hai14}. However it does not account for any of the behaviour of $H$ in time and as a result, the restriction in the second supremum that $|x-y| \leq \onorm{t,s} \defeq \onorm{t} \wedge \onorm{s}$ appearing in the homogeneous case has been removed. This turns out to be important for obtaining the relevant Schauder estimates; see \cite[Theorem 2.21]{HM15}. 

We then define the {\it inhomogeneous $\CD_T^{\gamma, \eta}$ norm} as

\begin{equ}[e:ModelledDistributionNorm]
	\VERT H \VERT_{\gamma, \eta; T} \defeq \Vert H \Vert_{\gamma, \eta; T} + \sup_{\substack{s \neq t \in (0,T] \\ | t - s | \leq \onorm{t, s}^{\s_0}}} \sup_{x \in \mathbb{R}^d} \sup_{l < \gamma} \frac{\| H_t(x) - \Sigma_x^{t s} H_{s}(x) \|_l}{\onorm{t, s}^{\eta - \gamma} |t - s|^{(\gamma - l)/\s_0}}.
\end{equ}

\begin{definition}
	We define $\CD_T^{\gamma, \eta}(Z)$ to be the space of functions $H: (0,T] \times \mathbb{R}^d \to \CT_{<\gamma}$ such that $\VERT H \VERT_{\gamma, \eta; T} < \infty$.
\end{definition}

\begin{remark}
	Inhomogeneous analogues of the usual Reconstruction Theorem and Schauder estimates hold true in this setting; see \cite[Theorems 2.11 and 2.21]{HM15}. Formulating a suitable fixed point result requires a little more care than in the homogeneous case, but this is also obtained in \cite{HM15}. The subject of the next subsection is a formulation of their construction specific to our setting.
\end{remark}

Given a second model $\bar{Z} = (\bar{\Pi}, \bar{\Gamma}, \bar{\Sigma})$ for $\ST$, we define the distance $\VERT H; \bar{H}\VERT_{\gamma, \eta; T}$ between $H \in \CD_T^{\gamma, \eta}(Z)$ and $\bar{H} \in \CD_T^{\gamma, \eta}(\bar{Z})$ by setting
\minilab{ModelledNorms}
\begin{equs}
	\| H; \bar{H} \|_{\gamma, \eta; T} &\defeq \sup_{t \in (0,T]} \sup_{x \in \mathbb{R}^d} \sup_{l < \gamma} \onorm{t}^{(l - \eta) \vee 0} \| H_t(x) - \bar{H}_t(x) \|_l \\
	& + \sup_{t \in (0,T]} \sup_{\substack{x \neq y \in \mathbb{R}^d \\ | x - y | \leq 1}} \sup_{l < \gamma} \frac{\| H_t(x) - \Gamma^{t}_{x y} H_t(y) - \bar{H}_t(x) + \bar{\Gamma}^{t}_{x y} \bar{H}_t(y) \|_l}{\onorm{t}^{\eta - \gamma} | x - y |^{\gamma - l}}\;,\\
	\VERT H; \bar{H} \VERT_{\gamma, \eta; T} &\defeq \| H; \bar{H} \|_{\gamma, \eta; T}\\
	& + \sup_{\substack{s \neq t \in (0,T] \\ | t - s | \leq \onorm{t, s}^{\s_0}}} \sup_{x \in \mathbb{R}^d} \sup_{l < \gamma} \frac{\| H_t(x) - \Sigma_x^{t s} H_{s}(x) -  \bar{H}_t(x) + \bar{\Sigma}_x^{t s} \bar{H}_{s}(x)\|_l}{\onorm{t, s}^{\eta - \gamma} |t - s|^{(\gamma - l)/\s_0}}\;.
\end{equs}

\subsection{A Truncated Regularity Structure for \eqref{e:Phi}}

It was shown in \cite[Section 8.1]{Hai14} that given a locally subcritical equation with nonlinearity of the form $F(u, \nabla u, \xi)$, one can construct a regularity structure $\ST = (\CT, \CG)$ and an associated space of modelled distributions such that the given equation can be formulated as a fixed point problem is that space. 

The equation \eqref{e:TPsi} does not quite fit immediately into that framework since its nonlinearity contains the non-local term $\beta \langle \Psi^{(n)}, \psi\rangle^3 \psi$. Instead, we consider the regularity structure $\ST$ constructed in \cite{Hai14} for \eqref{e:Phi} and suitably interpret the non-local term of \eqref{e:TPsi} in this regularity structure, allowing us to work with that fixed regularity structure for both equations.

We don't recall here the full details of the construction of \cite{Hai14}. Instead, we mention only that the construction includes the recursive construction of sets of symbols $\CF, \CU$ that one would expect to need to formulate the fixed point argument for \eqref{e:Phi}. 

The set $\CF$ contains the symbols required to describe terms appearing on the right hand side of equation \eqref{e:Phi}. In particular, $\CF$ is a subset of the set of symbols generated from $\{\r1,\Xr{i}, \AbSymXi\}$ under the operations $\Abtau \mapsto \Abint{\tau}$ and $(\Abtau,\absymbol{\bar{\tau}}) \mapsto \Abtau \absymbol{\bar{\tau}}$ and the assumptions that the multiplication is commutative with identity element $\r1$ and that $\Abint{X^k} = \absymbol{0}$ for each $k \in \mathbb{N}^3$.
In the case of \eqref{e:Phi}, the first few symbols of $\CF$ are
$$\CF = \{\r1, \AbSymXi, \<1r>, \<2r>, \<3r>, \<2r>\Xr{i}, \<31r>, \<32r>, \<22r>, \<20r>,\dots\}.$$

Here we have adopted the usual graphical notation of defining $\<1r> \defeq \Abint{\Xi}$ and then defining the remaining rooted trees recursively, where abstract integration is represented by drawing an edge downward from the root and multiplication is represented by concatenation of trees at the root.

Meanwhile, $\CU$ contains the symbols required to describe the solution of equation \eqref{e:Phi}. Concretely, we have that $\CU = \{\absymbol{X^k} : k \in \mathbb{N}^3\} \cup \absymbol{\mathcal{I}}(\mathcal{F})$. In the case of \eqref{e:Phi}, we have that $\mathcal{U} \subseteq \mathcal{F}$.
The regularity structure $\ST$ then has model space $\CT = \spanning \CF$ equipped with the grading $|\cdot|$ defined by
$$|\r1| = 0, \qquad |\Xr{i}| = 1, \qquad |\AbSymXi| = -\frac{5}{2}-\kappa, \qquad |\Abint{\tau}| = |\Abtau| + 2, \qquad |\Abtau \absymbol{\bar{\tau}}| = |\Abtau| + |\absymbol{\bar{\tau}}|$$
where $\kappa > 0$ is chosen sufficiently small.
At this point, the construction given in \cite[Section 8.1]{Hai14} yields a structure group $\CG$ acting on $\CT$ with the desired properties. We will not give details of the construction here but rather will recall the key properties as we need them.

Since we aim to formulate our equations as fixed point problems in a $\CD_T^{\gamma, \eta}$ space, we will not need the symbols beyond a suitably large homogeneity. Hence, for the rest of the paper we will work in the ambient regularity structure $\ST_r$
with model space $\CT_{<r}$ and structure group given by the restriction of $\CG$ to this space. It turns out that for our purposes, it will suffice to take $r = 2+\kappa$ where will choose $\kappa$ sufficiently small so that $\CT_{<r}$ contains only symbols of homogeneity at most $2$.

As is the case in \cite{HM15}, we cannot define a suitable inhomogeneous model on the entirety of the regularity structure $\ST_r$ since a typical choice of lift for space-time white noise will not be a function in time. To circumvent this problem, \cite{HM15} perform a truncation which removes the problematic symbols. For the general definition of truncations they are able to accommodate, see \cite[Definition 3.4]{HM15}. Here, we will only introduce the truncated structure that we require for our problem.

We define the sets $$\CF^{\gen} \defeq \{\<1r>, \<30r>\} \cup \CFpoly\;,\qquad \hat{\CF} = \CF \setminus \{\AbSymXi, \<3r>\}$$
where $\CFpoly = \{\r1, \Xr{i}, \Xr{i}\Xr{j}; i,j = 1,2,3\}$. 
From these sets, we define the {\it generating regularity structure} $\ST^{gen}$ and the {\it truncated regularity structure} $\hat{\ST}$ by setting 
$$ \CT^\gen \defeq \spanning \CF^\gen, \qquad \hat{\CT} \defeq \spanning\{\tau \in \hat{\CF}: |\tau| < r\}\;,$$
each with structure group given by the corresponding restriction of $\CG$. 
We note that $\CG$ leaves both model spaces invariant so that this definition makes sense.

The key feature of this truncation is that the truncated structure still contains the symbols necessary to formulate a suitable version of the fixed point problem for the equation at hand, is small enough so that it admits suitable models $Z = (\Pi, \Gamma, \Sigma)$ and the structure group $\CG$ leaves the truncated structure invariant, so that $(\Gamma, \Sigma)$ naturally extend to $\CT_{<r}$. In particular, this last point means that given a model $Z$ for the truncated structure and $f:\mathbb{R}^d \to \CT_{< \gamma}$, one can still make sense of the statement $f \in \CD_T^{\gamma,\eta}(Z)$ since the relevant norms only depend directly on the action of $(\Gamma, \Sigma)$.

\subsection{The \eqref{e:Psi} Equation}

To complete this section we formulate \eqref{e:TPsi} in the setting described earlier in this section. We remark that the equivalent program for \eqref{e:Phi} was already completed in \cite{HM15}. The modifications for \eqref{e:TPsi} are only minor, but since the nonlinearity is non-local we will briefly outline what needs to be done to accommodate it. 

Defining the maps 
\begin{equ}
	\hat F(\Abtau) = \<3r> - \CQ_{\leq 0}(\Abtau^3)\;, \qquad I_0 = \<1r> - \<30r>\;,
\end{equ}
the abstract analogue of \eqref{e:Phi} is given as in \cite{HM15} by the fixed point problem
\begin{equ}\label{e:AbPhi}
	U = \CP\hat{F}(U) + S\Phi_0 + I_0\;,
\end{equ}
where $\CP \defeq \CK_{\bar{\gamma}} + R_\gamma \CR$ is the usual lift of the action of the heat kernel  and $S\Phi_0$ denotes the lift of the solution of the heat equation with initial condition $\Phi_0$ to the polynomial part of the regularity structure. We refer the reader to equations (2.29), (3.2) and Theorem~2.11 of \cite{HM15} for more details on the definitions of the operators 
appearing on the  right-hand side of this equation.
Formulating the fixed point problem in this way is advantageous since $\hat{F}$ has the property that if $\CT_\CU = \spanning \CU$ then $\hat{F}(\CT_\CU \cap \hat{\CT}) \subseteq \hat{\CT}$.

In order to formulate the corresponding analogue of \eqref{e:TPsi}, we introduce for $n \in \NBar$ the maps $G_n: \CD_T^{\gamma, \eta} \to \CD_T^{\gamma, \eta}$ given by
\begin{align*}
G_n(U) = \begin{cases}
\langle \CR_t U, \psi \rangle^3 \psi \r1 \qquad& n = \infty \\ F_n'(\langle \CR_t U, \psi \rangle) \psi \r1 \qquad& n < \infty
\end{cases} 
\end{align*}
where $\CR_t$ is the reconstruction operator at time $t$ -- see \cite[Thm~2.11]{HM15}. Here we begin to reap the benefits of working in the inhomogeneous setting since $ \langle \CR_t U, \psi \rangle$ is automatically well defined even though the testing is only in space; this would not be automatic in the setting of \cite{Hai14}.

We then define the abstract analogue of \eqref{e:TPsi} ($n \in \NBar$) to be
\begin{equs}
	\label{e:TAbPsi}
	U_n = \CP(\hat{F}(U_n) + \beta G_n(U_n)) + S\Phi_0 + I_0.
\end{equs}

In this setting, one has the following analogue of \cite[Thm~3.10]{HM15}. 

\begin{theorem}\label{t: Fixed Point}
	Let $\gamma = 1+\kappa$ and $\eta \leq \alpha$. Then for any model $Z$ with time regularity $\eta > 0$ on $\hat{\ST}$ and for every periodic $\Phi_0 \in \CC^\eta(\mathbb{R}^d)$, there exists a time $T_* \in (0,\infty]$ such that for every $T < T_*$ the equations \eqref{e:AbPhi}, \eqref{e:TAbPsi} admit a unique periodic solution $U \in \CD_T^{\gamma, \eta}(Z)$ for all $n \in \NBar$. Furthermore, if $T_* < \infty$ then either
	\begin{align*}
	\lim_{T \uparrow T_*} \|\CR_TS_T^\Phi(\Phi_0, Z)_T\|_{\CC^\eta} = \infty
	\end{align*}
	or there is an $n \in \NBar$ such that
	\begin{align*}
	\lim_{T \uparrow T_*} \|\CR_TS_T^{\Psi_n}(\Phi_0, Z)_T\|_{\CC^\eta} = \infty
	\end{align*}
	where $S_T^\Phi, S_T^{\Psi_n}: (\Phi_0, Z) \mapsto U$ are the solution maps for \eqref{e:AbPhi} and \eqref{e:TAbPsi} respectively. 
	
	Additionally, for every $T < T_*$, the solution maps $S_T^\Phi, S_T^{\Psi_n}$ are jointly Lipschitz continuous in a neighbourhood around $(\Phi_0, Z)$ in the sense that, for any $B > 0$ there is $C > 0$ such that, if $\bar{U} = \mathcal{S}_T^\Phi(\bar{\Phi}_0, \bar{Z})$ for some initial data $(\bar{\Phi}_0, \bar{Z})$ where $\bar{Z}$ has time regularity $\eta > 0$, then one has the bound $\VERT U; \bar{U} \VERT_{\gamma, \eta; T} \leq C \eta$, provided $\Vert \Phi_0 - \bar{\Phi}_0 \Vert_{\CC^{\eta}} + \VERT Z; \bar{Z} \VERT_{\gamma; T} \leq \eta$, for any $\eta \in (0,B]$ and similarly for $S_T^{\Psi_n}$.
\end{theorem}
\begin{proof}
	The result for the case of \eqref{e:Phi} follows almost exactly as in the fixed point argument of \cite{Hai14} (see Theorem~7.8 and Proposition~7.11 there) and is the content of \cite[Thm~3.10]{HM15}. Since we have a direct interpretation of the additional nonlinearity of \eqref{e:TPsi} in the polynomial part of our regularity structure, the result for that equation will follow from the same techniques with no difficulties as soon as we verify that for every $R > 0$ there exists $C$ such that for every $H, \bar{H}$ such that $\VERT H \VERT_{\eta, \gamma; T} + \VERT\bar{H} \VERT_{\eta, \gamma; T} \leq R$, $G_n$ satisfies
	\begin{equ}
		\VERT G_n(H), G_n(\bar{H}) \VERT_{\bar{\gamma}, \eta - \bar{\gamma}; T} \leq C \left( \VERT H, \bar{H} \VERT_{\gamma, \eta;T} + \VERT Z, \bar{Z} \VERT_{\gamma;T} \right)
	\end{equ}
	for all $n \in \NBar$ and $H, \bar{H}$ for some $0 < \bar{\gamma} \ll \eta$. Since $\psi \in \CB_0^r$, this is immediate from the regularity of the reconstruction operator provided by the bound (2.13) of \cite[Thm~2.11]{HM15}, combined with smoothness of $F_n'$ and $x \mapsto x^3$.
\end{proof}

\begin{remark} \label{r: classical solutions}
	If the model $Z$ in the preceding theorem is the canonical lift of a smooth driving noise as in \cite[Sec.~8.2]{Hai14} then the reconstruction operator is given by 
	\begin{equ}
		\left(\CR_t U_t\right)(x) = \left(\Pi_x^t U_t(x)\right)(x)\;.
	\end{equ} 
	In this case, the reconstruction of the abstract solutions found above coincides with the classical solutions of \eqref{e:Phi}, \eqref{e:Psi} with smooth driving noise and no renormalisation.
	
	Additionally, if we define the renormalisation map $M$ as in \cite[Sec.~9.2]{Hai14} with constants $C_1$ and $C_2$ and build a renormalised smooth model $Z^M$ as in \cite[Sec.~8.3]{Hai14}, then $u_\Phi \defeq \CR_tS_t^\Phi(\Phi_0, Z)_t$ and $u_{\Psi_n} \defeq \CR_tS_t^{\Psi_n}(\Psi_0, Z)_t$ solve the following equations with smooth driving noise (cf.~\cite[Prop.~9.10]{Hai14})
	\begin{equs}
		\partial_t u_\Phi =& \Delta u_\Phi - u_\Phi^3 + (3C_1 - 9C_2)u_\Phi + \xi,
		\\
		\partial_t u_{\Psi_n} =& \Delta u_{\Psi_n} - u_{\Psi_n}^3 + (3C_1 - 9C_2)u_{\Psi_n} + \beta F_n'(\langle u_{\Psi_n}, \psi \rangle) \psi + \xi.
	\end{equs}
	Finally, if $\xi_\varepsilon = \rho_\varepsilon \ast \xi$ where $\rho_\varepsilon$ is a mollifier at scale $\varepsilon$, then there is a choice of diverging constants $C_1^\varepsilon, C_2^\varepsilon$ such that the solutions to the above equations converge in probability as $\varepsilon \to 0$ (cf.~\cite{ChHa16}). We define the solution of \eqref{e:Phi}, \eqref{e:TPsi} to be these limits which are independent of the choice of mollifier.
\end{remark}

\begin{remark}\label{r:NoBlowUp}
	It is known that \eqref{e:Phi} has a `coming down from infinity' property that precludes the blow-up of $\CC^\alpha$ norms of solutions (cf.~\cite{MW18a}). Section~\ref{s: Cont Bounds} adapts the techniques of \cite{MW18a} to show the corresponding result for \eqref{e:TPsi}. As a result, in the cases of interest for us it follows that in the setting of Theorem~\ref{t: Fixed Point} 
	one actually has $T_* = \infty$. 
\end{remark}


\section{Discrete Inhomogeneous Models}\label{s:Dreg}


In this section we introduce the discrete analogues of the objects and the results of the last section. We will use the discretisations to identify the density of the invariant measure of \eqref{e:TPsi} with respect to \eqref{e:Phi} for $n \in \mathbb{N}$. In particular, the goal is to treat for an arbitrary fixed $n \in \mathbb{N}$ discretisations of \eqref{e:Phi}, \eqref{e:TPsi} of the form 
\begin{equs}\label{e:DPhiRenorm}
	\frac{d}{dt} \Phi^\varepsilon =& \Delta^\varepsilon \Phi^\varepsilon + C^{(\varepsilon)} \Phi^\varepsilon -\left(\Phi^\varepsilon \right)^3 + \xi^\varepsilon\;, \quad \Phi^\varepsilon(0, \cdot) = \Phi^\varepsilon_0(\cdot)\;, \tag{$\Phi^{4}_{3,\varepsilon}$}
	\\ \label{e:DPsiRenorm}
	\frac{d}{dt} \Psi^\varepsilon =& \Delta^\varepsilon \Psi^\varepsilon + C^{(\varepsilon)} \Psi^\varepsilon -\left(\Psi^\varepsilon \right)^3 + \beta F_n'(\langle \iota^\varepsilon \Psi^\varepsilon, \psi \rangle) \psi^\varepsilon + \xi^\varepsilon\;, \quad \Psi^\varepsilon(0, \cdot) = \Psi^\varepsilon_0(\cdot)\;, \tag{$\Psi^{4,n}_{3,\varepsilon}$}
\end{equs}
on the discretisation $\DTorus$ of $\Torus$ with grid scale $\varepsilon = 2^{-N}$ for $N \in \mathbb{N}$, where $\Phi^\varepsilon_0, \Psi^\varepsilon_0 \in \mathbb{R}^{\Torus_\varepsilon}$, $\Delta^\varepsilon$ is the nearest-neighbour approximation of the Laplacian $\Delta$ and $\xi^\varepsilon$ are spatial discretisations of $\xi$ defined on a single common probability space by setting
\begin{equ}[e:SimpleDNoise]
	\xi^\varepsilon(t,x) \defeq \varepsilon^{-3} \langle \xi(t, \cdot), \1_{\eBox{x}} \rangle\;, \qquad (t,x) \in \mathbb{R} \times \Torus_\varepsilon\;,
\end{equ} 
where, for $x \in \Torus_\varepsilon$, $\eBox{x} \subset \Torus$ denotes the cube of side length
$\varepsilon$ centred at $x$.

The function $\psi^\varepsilon \in \mathbb{R}^\DTorus$ is defined by 
$$\psi^\varepsilon(y) = \varepsilon^{-3} \int_\eBox{y} \psi(z) dz.$$
Finally $C_\lambda^{(\varepsilon)} \sim \varepsilon^{-1} + \log \varepsilon$ is a sequence renormalisation constants for which precise values are given in \cite[Eq.~7.6]{HM15} and 
the subsequent paragraph.

Of course, equations \eqref{e:DPhiRenorm} and \eqref{e:DPsiRenorm} are nothing but SDEs with global in time solutions. However, in order to prove the convergence to the continuum solutions, it is useful to recast their solution theory in the language of regularity structures. 
We recall the following definitions from \cite{HM15}. 

\begin{definition}
	Given $\varepsilon > 0$, a {\it discrete model} at grid-scale $\varepsilon$ for the regularity structure $\ST$ consists of maps $(\Pi^\varepsilon, \Gamma^\varepsilon, \Sigma^\varepsilon)$
	\begin{equ}
		\Pi_x^{\varepsilon, t}: \CT \to \mathbb{R}^{\DTorus}\;, \qquad \Gamma^{\varepsilon, t} : \DTorus \times \DTorus \to \CG\;, \qquad \Sigma^{\varepsilon}_x : \mathbb{R}\times \mathbb{R} \to \CG\;,
	\end{equ}
	indexed by $t \in \mathbb{R}$ and $x \in \DTorus$, which have all the algebraic properties of their continuous counterparts in Definition~\ref{d:Model}, with the spatial variables restricted to the grid $\DTorus$. Additionally, we assume that $\bigl(\Pi^{\varepsilon, t}_x \tau\bigr) (x) = 0$, for all $\tau \in \CT_l$ with $l > 0$, and all $x \in \DTorus$ and $t \in \mathbb{R}$.
\end{definition}

The seminorms $\| \Pi^\varepsilon \|^{(\varepsilon)}_{\gamma; T}$ and $\| \Gamma^\varepsilon \|_{\gamma;T}^{(\varepsilon)}$ are defined to be the smallest constants such that the inequalities \eqref{e:PiGammaBound} hold uniformly in $\lambda \in [\varepsilon,1]$, $x, y \in \DTorus$, $t \in \mathbb{R}$ and with the usual duality pairing of $\CD'(\mathbb{R} \times \Torus) \times \CD(\mathbb{R} \times \Torus)$ replaced with the discrete pairing 
$$\langle F, \varphi \rangle_\varepsilon \defeq \int_\mathbb{R} \sum_{y \in \DTorus} F(t,y) \varphi(t,y) dt.$$
The quantity $\| \Sigma^\varepsilon \|_{\gamma;T}^{(\varepsilon)}$ is then defined as the smallest constant $C$ such that the bounds
\begin{equ}[e:DSigmaBound]
	\| \Sigma^{\varepsilon, s t}_{x} \tau \|_{m} \leq C \| \tau \| \bigl(|t - s|^{1/\s_0} \vee \varepsilon\bigr)^{l - m}\;,
\end{equ}
hold uniformly in $x \in \DTorus$ and the other parameters as in \eqref{e:SigmaBound}.

We measure the time regularity of $\Pi^\varepsilon$ as in \eqref{e:PiTimeBound}, by substituting the continuous objects by their discrete analogues, and by using $|t - s|^{1/\s_0} \vee \varepsilon$ instead of $|t - s|^{1/\s_0}$ on the right-hand side. We also define quantities $\| \cdot \|^{(\varepsilon)}$, $\VERT \cdot \VERT^{(\varepsilon)}$, in the same way as the above construction that measure the size of (resp.\ distance between) model(s) $Z$ (resp.\ $Z, \bar{Z}$).

\begin{remark}
	The pairing $\langle \cdot, \cdot \rangle_\varepsilon$ does not correspond to the embedding $\iota^\varepsilon$. Indeed, it does not correspond to any embedding $e: \mathbb{R}^\DTorus \to C^{-\frac12 -}$ since the action in space is that of a Dirac delta which has regularity no better than $-3$. This is not a serious issue for us since the difference between the two ways of testing applied with the solutions of \eqref{e:DPhiRenorm}, \eqref{e:DPsiRenorm} converges to $0$ as $\varepsilon \to 0$.
\end{remark}

One then has the following discrete analogue of the $\CD_T^{\gamma, \eta}$ spaces. For $\gamma, \eta \in \mathbb{R}$, a fixed time $T > 0$ and a discrete model $Z^\varepsilon=(\Pi^\varepsilon, \Gamma^\varepsilon, \Sigma^\varepsilon)$ on a regularity structure $\ST$, we define the $\CD_{T,\varepsilon}^{\gamma, \eta}$ norms $\| H^\varepsilon \|^{(\varepsilon)}_{\gamma, \eta; T}$ and
$\VERT H^\varepsilon \VERT^{(\varepsilon)}_{\gamma, \eta; T}$ of a function $H^\varepsilon : (0, T] \times \DTorus \to \CT_{<\gamma}$ in exactly the same way as in 
\eqref{e:ModelledDistributionNormSpace} and \eqref{e:ModelledDistributionNorm}, except that the spatial variables run over $\DTorus$ and
the powers of $|t|_0$ and $|t,s|_0$ appearing there are replaced by
 $\onorm{t} \vee \varepsilon$ and $|t,s|_0 \vee \varepsilon$ respectively.

\begin{definition}
	$\CD_{T, \varepsilon}^{\gamma, \eta}$ is the space of functions $H^\varepsilon : (0, T] \times \DTorus \to \CT_{<\gamma}$ such that $\VERT H^\varepsilon \VERT_{\gamma, \eta;T}^{(\varepsilon)} < \infty$.
\end{definition}

\begin{remark}
	In the setting of discrete inhomogeneous models, suitable instances of the usual results in the theory of regularity structures hold. For example, one has a reconstruction operator with the explicit representation $(\CR_t^\varepsilon H_t^\varepsilon)(x) \defeq (\Pi_x^{\varepsilon,t}H_t^\varepsilon(x))(x)$ \cite[Thm~4.6]{HM15}. Additionally, the Green's function for the discretised heat equation has a decomposition that is a suitable analogue of the decomposition of the heat kernel given in \cite{Hai14} (see \cite[Lem.~5.4]{HM15}) and the corresponding lift of the action of the kernel satisfies analogues of the usual Schauder estimates, \cite[Thm~4.17]{HM15}.
\end{remark}

We now obtain the solutions of \eqref{e:DPhiRenorm}, \eqref{e:DPsiRenorm} from an abstract fixed point argument. We will eventually handle the renormalisation terms containing factors $C_\lambda^{(\varepsilon)}$ at the level of our choice of model so that the abstract formulation of these equations will be given by 
\begin{equs}
	\label{e:AbDPhi} U^\varepsilon =& \CP^\varepsilon \hat{F}(U^\varepsilon) + S^\varepsilon \Phi_0^\varepsilon + I_0 \\
	\label{e:AbDPsi} U_n^\varepsilon =& \CP^\varepsilon (\hat{F}(U_n^\varepsilon) + \beta G_n^\varepsilon(U_n^\varepsilon)) + S^\varepsilon \Phi_0^\varepsilon + I_0
\end{equs}
where $S^\varepsilon \Phi_0^\varepsilon$ is the solution to the semidiscrete heat equation with initial condition $\Phi_0^\varepsilon \in \mathbb{R}^\DTorus$ and $\CP^\varepsilon = \CK_{\bar{\gamma}}^\varepsilon + R_\gamma^\varepsilon \CR^\varepsilon$ is the abstract analogue of the action of the semidiscrete heat kernel (see equations $(4.24)$ and $(5.17)$ of \cite{HM15} for the definitions of $R_\gamma^\varepsilon$ and $\CK_{\bar{\gamma}}^\varepsilon$). Finally
\begin{align*}
G_n^\varepsilon(U^\varepsilon)(t,x) \defeq \begin{cases}
F_n'(\langle \iota^\varepsilon \CR_t^\varepsilon U^\varepsilon , \psi \rangle) \psi^\varepsilon(x), & n \in \mathbb{N} \\
\langle \iota^\varepsilon \CR_t^\varepsilon U^\varepsilon , \psi \rangle^3 \psi^\varepsilon(x), & n = \infty
\end{cases}
\end{align*}
One then has the following analogue of Theorem~\ref{t: Fixed Point}, which is essentially special case of \cite[Thm~5.8]{HM15} up to the minor adaptation required to accommodate the nonlinearity $G_n^\varepsilon$ which is similar to that performed in Section~\ref{s:reg} and so we omit the details.

\begin{theorem}\label{t:D Fixed Point}
	Let $Z^\varepsilon$ be a sequence of discrete inhomogeneous models indexed by $\varepsilon = 2^{-N}$ for $N \geq 1$. Then for every $T< \infty$, the sequence of solution maps $\mathcal{S}^\varepsilon_T : (\Phi_0^\varepsilon, Z^\varepsilon) \mapsto U^\varepsilon$ of the equation \eqref{e:AbDPhi} up to time $T$ is jointly Lipschitz continuous (uniformly in $\varepsilon$) in the sense of Theorem~\ref{t: Fixed Point}, but replacing the continuum objects with their discrete analogues. The same is true of the solution map for \eqref{e:AbDPsi}.
\end{theorem}

In order to state our direct analogue of the main convergence result of \cite{HM15} we introduce a choice of discretisation of initial condition $\phi \in C^\alpha(\Torus)$. We recall that we defined
$$\FP_\varepsilon \phi \defeq \langle \phi, \varepsilon^{-3} 1_{\{\|\cdot - x\|_\infty \leq \frac{\varepsilon}{2}\}} \rangle \in L^\infty(\DTorus).$$
The right hand side of this expression is well defined for $\phi \in C^{-1+\kappa}$ since the indicator function of a cube lies in the Besov space $\CB_{1,1}^{1-\kappa}$. One can then show (essentially by calculation; see Corollary \ref{c:init discretisation}) that for $\zeta \in C^{\alpha'}$, and $\zeta^\varepsilon \defeq \FP_\varepsilon \zeta$, $$\| \zeta ; \zeta^\varepsilon \|_{- \frac12 - \kappa}^{(\varepsilon)} \defeq \sup_{ x \in \DTorus} \sup_{\lambda \in (\varepsilon, 1]} \sup_{\psi \in \CB_r^0} \lambda^{\frac12 + \kappa} | \scal{\zeta - \iota^\varepsilon \zeta^\varepsilon, \psi_x^\lambda}| \to 0$$
as $\varepsilon \to 0$.

\begin{theorem}\label{t:Convergence}
	Let $\xi$ be a space-time white noise over $L^2(\mathbb{R} \times \Torus)$ on a probability space $(\Omega, \SF, \mathbf{P})$ and let $\Phi$ and $\Psi$ be the unique maximal solution of \eqref{e:Phi} and \eqref{e:TPsi} respectively with initial condition $\phi \in C^{\alpha}(\Torus)$. Let $\xi^\varepsilon$ be given by \eqref{e:SimpleDNoise}, and let $\Phi^\varepsilon$ and $\Psi^\varepsilon$ be the unique global solution of \eqref{e:DPhiRenorm} and \eqref{e:DPsiRenorm} respectively with initial condition $\FP_\varepsilon \phi$.
	Then there exists a sequence of stopping times $T_\varepsilon$ such that $\mathbb{P}(T_\varepsilon < T) \to 0$ as $\varepsilon \to 0$ for any fixed $T$ positive and
	\begin{equ}
		\sup_{t \in [0,T_\varepsilon]} \|\Phi(t) - \iota^\varepsilon \Phi^\varepsilon(t)\|_{C^{-\frac12 -}} \to 0\;,\qquad
		\sup_{t \in [0,T_\varepsilon]} \|\Psi(t) - \iota^\varepsilon \Psi^\varepsilon(t)\|_{C^{-\frac12 -}} \to 0\;,
	\end{equ}
	in probability as $\varepsilon \to 0$. Furthermore the above convergence is locally uniform in the initial condition $\phi$.
\end{theorem}
\begin{proof}
	This follows from the same techniques as the proof of \cite[Thm~1.1]{HM15} for the case of \eqref{e:Phi}. This proof proceeds by using Theorem~\ref{t: Fixed Point} and Theorem~\ref{t:D Fixed Point} along with convergence of suitable Gaussian models to reduce to the convergence in the case of a smooth driving noise which is a problem of numerical analysis. 
	
	We note that the same techniques work here since we have formulated the abstract version of \eqref{e:TPsi} and its discrete counterpart on the same regularity structure as those for \eqref{e:Phi} such that the equations are simultaneously driven by the same choice of model. Hence following the proof of \cite[Thm~1.1]{HM15} yields the convergence
	\begin{equs}
		\sup_{t \in [0,T_\varepsilon^1]}\left( \|\Phi(t); \Phi^\varepsilon(t)\|_{C^{-\frac12-}}^{(\varepsilon)} \vee \|\Psi(t); \Psi^\varepsilon(t)\|_{C^{-\frac12-}}^{(\varepsilon)} \right) \to 0\;,
	\end{equs}
	in probability as $\varepsilon \to 0$, where $\|\zeta; \zeta^\varepsilon\|_{C^{\alpha}}^{(\varepsilon)} \defeq \sup_{\varphi \in \CB_0^r} \sup_{x \in \DTorus} \sup_{\lambda \in [\varepsilon, 1]} \lambda^{\alpha} |\langle \zeta, \varphi_x^\lambda \rangle - \langle \zeta^\varepsilon, \varphi_x^\lambda \rangle_\varepsilon|$ and $T_\varepsilon^1$ is a suitable subsequence of $H_K \vee \tilde{H}_K$ where $H_K$ (resp. $\tilde{H}_K$) is the exit time of the ball of radius $K$ for $\Phi$ (resp. $\Psi$).
	
	Unfortunately, as mentioned earlier, the discrete testing appearing here is the wrong one. Since the behaviour below scale $\varepsilon$ is straightforward, it remains to see that 
	$$\sup_{\varphi \in \CB_0^r} \sup_{x \in \DTorus} \sup_{\lambda \in [\varepsilon, 1]} \lambda^\alpha |\langle \iota^\varepsilon \Phi^\varepsilon, \varphi_x^\lambda \rangle - \langle \Phi^\varepsilon, \varphi_x^\lambda \rangle_\varepsilon| \to 0$$
	and similarly for $\Psi^\varepsilon$. Notice that we can write
	\begin{equ}
		\langle \iota^\varepsilon \Phi^\varepsilon, \varphi_x^\lambda \rangle - \langle \Phi^\varepsilon, \varphi_x^\lambda \rangle_\varepsilon = \sum_{y \in \DTorus} \Phi^\varepsilon(y) \int_\eBox{y} \left( \varphi_x^\lambda(z) - \varphi_x^\lambda(y)\right)dz \les \lambda^{-1} \varepsilon \|\Phi^\varepsilon\|_\infty
	\end{equ}
	since $|\varphi_x^\lambda(y) - \varphi_x^\lambda(z)| \leq \lambda^{-4} |z - y|$ and the summand is non-zero only for $y$ such that $|y-x| \leq 2\lambda$. Hence, it suffices to see that $\|\Phi^\varepsilon\|_\infty \les \varepsilon^{\alpha}$.
	
	For this we only sketch the details, since they are an application of the same discrete tools as used repeatedly in this paper and in \cite{HM15}. Indeed, this is a corollary of the same rate of blow up for discretisations of the stochastic heat equation since as usual for \eqref{e:Phi}, $\Phi = u + v$ where $u$ is the solution of the stochastic heat equation and $v \in C^{\frac12 - \kappa}$ is the solution of a `remainder equation'. One has a similar decomposition for $\Phi^\varepsilon$ and the techniques used above yield that 
	$$ \sup_{t \in [0,T_\varepsilon^2]} \|v^\varepsilon(t) - v(t)\|_\infty \to 0$$ for some sequence of stopping times $T_\varepsilon^2$ satisfying $\mathbb{P}(T_\varepsilon^2 < T) \to 0$ as $\varepsilon \to 0$ for every fixed $T > 0$. 
	In particular, taking $T_\varepsilon = T_\varepsilon^1 \wedge T_\varepsilon^2$ yields the desired result since then the above yields control on the supremum norm of $v^\varepsilon$ so that the only blow-up in $\Phi^\varepsilon$ comes from $u$.
\end{proof}

\section{Bounds for the Continuum $\Psi_3^{4,n}$ equation}\label{s: Cont Bounds}

In this short section, we state an a priori bound which is uniform in $n$ and $\eps$ for the PDE with smooth driving noise
\begin{align}\label{e:smooth Psi renorm}
(\partial_t - \Delta)u = -u^3 +(3C_1^{(\eps)} - 9C_2^{(\eps)})u + \beta \psi F_n'(\langle u, \psi \rangle) + \xi_\eps
\end{align}
that is a direct adaptation of the main result of \cite{MW18a} which give the equivalent bound for the solution of \eqref{e:Phi}. (Take for example $\xi_\eps = \rho_\eps \star \xi$ for $\rho_\eps$
a smooth compactly supported mollifier.) The constants $C_i^{(\eps)}$ are the 
same renormalisation constants as in the proof of \cite[Thm~10.22]{Hai14}; in particular they do not depend
on the additional nonlinearity appearing on the right-hand side.
This kind of bound is of interest to us since the terms appearing on the right-hand side of our a priori bound will converge to natural limiting objects as $\varepsilon \to 0$ so that these bounds will directly transfer to the solution of \eqref{e:TPsi}.

Since we are able to restart the equation, it is sufficient for us to obtain good bounds up to time $1$. Hence we define the cylinder $P \defeq (0,1) \times \Torus$ and for $R > 0$ we set $P_R = (R^2,1) \times \Torus$. The bounds we obtain will also depend on $R$ in an explicit way which will enable the bounds to be independent of the choice of initial condition.

We draw attention here to the fact that our choice of spatial domain and the resulting definition of $P_R$ differ slightly to the ones given in \cite{MW18a}. This is necessary since our additional non-linearity is non-local in space and so to adapt their argument we must work in a setting where localisation in space is not included in the proof. 

In particular, the key a priori bound on which the proof technique of \cite{MW18a} is premised (their Lemma 2.7) must be slightly adapted to accommodate this shift in setting. A trivial adaptation of their proof of this Lemma yields the following result.

\begin{lemma}\label{l: maximum estimate}
	Let $u$ be a continuous function defined on $[0,1] \times \mathbb{T}^3$ such that 
	$$(\partial_t - \Delta)u(z) = - u(z)^3 + g(z,u)$$
	pointwise for $z \in (0,1] \times \mathbb{T}^3$ for a bounded function $g$. Then for $z = (t,x) \in (0,1] \times \mathbb{T}^3$, one has that
	$$|u(z)| \leq C \max\{t^{-\frac12}, \|g\|^\frac13\}$$ for some independent constant $C$.
\end{lemma}

If $\<1>$ is a solution of the equation $(\partial_t - \Delta) \<1> = \xi_\eps$ then the techniques of \cite{MW18a} will in fact yield supremum norm bounds on $v \defeq u - \<1>$ where $u$ is the solution of \eqref{e:smooth Psi renorm}.

First, we introduce the definitions of some graphical notation appearing in \cite{MW18a}. We emphasise that the trees 
appearing in this section are coloured black since they are not elements of a regularity structure. They also slightly 
differ from the BPHZ model applied to those elements since they are constructed directly from the PDE, rather than 
via convolution with some cut-off version of the heat kernel.

We define $\<2> \defeq \<1>^2 - C_1^{(\eps)}$ and $\<3> \defeq \<1>^3 - 3C_1^{(\eps)} \<1>$, leaving the $\eps$-dependence 
implicit. We then introduce the higher order symbols $\<20>, \<30>$ which are assumed to satisfy 
$$(\partial_t - \Delta) \<20> = \<2>, \; \quad (\partial_t - \Delta) \<30> = \<3>.$$

For $\alpha > 0$, we let $[\cdot]_{\alpha, C}$ be the usual $\alpha$-Hölder seminorm restricted to points in the set $C \subseteq \mathbb{R} \times \Torus$. If $C$ is omitted, it is to be understood that it is the whole space. 

To define an analogue of these seminorms for $\alpha < 0$, we fix a family of smooth compactly supported test functions $\phi_T\colon \R\times \R^{3}\to \R$ with a semigroup property at dyadic scales as constructed in \cite[Sec.~2]{MW18a}. The precise form of $\phi_T$ won't matter to us except that it is required to prove the analogue of the reconstruction theorem \cite[Thm~2.8]{MW18a} used in the paper of Moinat and Weber which is implicitly also required here. Since we do not retrace many details of the proofs of \cite{MW18a} in this section, we refer the interested reader to that paper for more details.

Having introduced this quantity, we now introduce a finite collection 
\begin{equ}
\fT = \{\<1>,\<2>,\<2>_x,\<20>,\<3>,\<30>, \<22>, \<31>,\<32>\}\;,
\end{equ}
of higher order trees. Each of these trees represents
an $\eps$-dependent random function (it actually depends furthermore on a space-time ``base point''
since we consider the ``positively renormalised'' quantities). Details of their construction
do not matter for the purpose of this discussion, but we introduce the quantity $n_\tau$
counting the number of leaves of a tree $\tau$ (so $n_{\<2>} = 2$, $n_{\<22>} = 4$, etc)
as well as a (random) quantity $[\tau]_\kappa$ measuring the size of these functions in 
a terms of how their convolution with $\phi_T$ behaves at the base point as $T \to 0$. 
(If the functions are base point independent, as is the case for example for \<1> and \<2>,
then these are equivalent to a Hölder norm of order $\deg \tau - \kappa$, where $\deg \<1> = -\frac12$, the degree is multiplicative, and solving the heat equation increases degree by $2$.)
See \cite[Eq.~2.13--2.19]{MW18a} for details of these definitions (note that their integer
multiples of $\eps$ are replaced by $\kappa$ in our notation). One has for example 
\begin{equs}{}
[\<2>]_\kappa &= \sup_{x \in P}\sup_{T< 1} T^{1+\kappa} \left|\int \<2>(y)\phi_T(y-x)\,dy\right|\;, \\{}
[\<22>]_{\kappa} &= \sup_{x \in P} \sup_{T< 1} T^{\kappa} \left | \int \left( \<20>(y)\<2>(y) - C_2^{(\eps)} - \<20>(x) \<2>(y) \right) \phi_T(y-x)\, dy \right|\;.
\end{equs}
Finally, since the tree $\<1>$ naturally plays a distinguished role in the equation for $u - \<1>$ because of the non-local term in the nonlinearity, we measure its regularity in a slightly stronger norm than \cite{MW18a} in order to get good bounds, and we set
$$[\<1>]_\kappa = \sup_{t \in [0,1]} \|\<1>(t, \cdot)\|_{C^{-\frac12-\kappa}(\Torus)}\;.$$
\begin{theorem}\label{t:continuum bound}
	Fix a smooth function $\psi: \mathbb{T}^3 \to \mathbb{R}$ such that $\|\psi\|_\infty, \|D\psi\|_\infty \leq 1$ and fix also $\beta > 0$ and $\kappa > 0$ small enough. If $u$ is the solution of \eqref{e:smooth Psi renorm} for this $\psi$ and $v = u - \<1>$, then for all $R \in (0,1)$ one has the bound
\begin{equ}[e:boundvfinal]
\|v\|_{P_R} \le C\max\Big\{ R^{-1},[\tau]_{\kappa}^\frac{2}{n_\tau (1-\kappa)}; \tau \in \fT\Big\}\;,
\end{equ}
	where $n_\tau$ is the number of leaves appearing in $\tau$. Here $C$ is a constant that is independent of $n$ and $\eps$, and $\|v\|_{P_R}$ denotes the supremum norm over $P_R$.
\end{theorem}
\begin{proof}
This follows with only very minor modifications of the proof of \cite[Thm~2.1]{MW18a}. Indeed, once one replaces applications of \cite[Lemma 2.7]{MW18a} with the equivalent application of our Lemma \ref{l: maximum estimate},
one only has to make adjustments to deal with the extra term appearing in the non-linearity of 
our equation. This only requires small changes in Section~4.2 of that paper since, once
one derives a similar bound to that given in the conclusion of that subsection, one can proceed 
with the rest of the proof with no significant changes. 
	
	The structure of the proof there is to assume that the bound 
\begin{equ}[e:contrav]
	\|v\|_{P_R} \leq 1 \vee C \max\Big\{[\tau]_{\kappa}^\frac{2}{n_\tau (1 -\kappa)}; \tau \in \fT\Big\}\;,
\end{equ}
	fails on some parabolic cylinder with a constant $C$ that depends only on combinatorial factors arising during their proof and then derive from the converse inequality a bound of order $R^{-1}$,
thus yielding \eqref{e:boundvfinal}. 
	
There are only two steps in the proof which rely on the precise form of the equation 
under consideration and not just on the local expansion of the solution up to order $1$ 
(which has the precise same form for \eqref{e:TPsi} as for \eqref{e:Phi}).
The first step, which is  given in their Section~4.2, is to consider
the two-parameter function $U$ given by
\begin{equ}
U(x,y) = v(y) - v(x) + \<30>(y) - \<30>(x) + 3v(x) \bigl(\<20>(y)-\<20>(x)\bigr)\;,
\end{equ}
see \cite[Eq.~2.29]{MW18a}, to fix an open space-time domain $D = P_R$ (for some $R > 0$),
and to assume a bound of the type
$[\tau]_{\kappa} \le c \|v\|_D^{n_\tau (1-\kappa)/2}$ for some
$c \le 1$. Writing $(\cdot)_T$ for space-time convolution with $\phi_T$, one then shows that
 there exists $\gamma \ge 0$ such that a bound of the form
\begin{equ}[e:wantedBoundU]
T^2\|\big((\partial_t - \Delta)U(x,\cdot)\big)_T\|_{B(x,L)} \lesssim  \|v\|_D (L/T)^\gamma\;,
\end{equ}
holds uniformly over all choices of $D$ with diameter bounded by $1$, all $x \in D$, and all 
$T \le L \le 1\wedge \|v\|_D^{-1}$ such that $B(x,2L) \subset D$, where $B(x,2L)$
denotes the \textit{parabolic} ball of radius $2L$ ``directed towards the past'', see
\cite[Eq.~2.2]{MW18a}.

\begin{remark}
The bound given in \cite[Eq.~4.20]{MW18a} appears stronger because of the presence
of the constant $c$ which can be made arbitrarily small. That bound however is incorrect since
the first term in \cite[Eq.~4.8]{MW18a} does not satisfy it. Fortunately, this
additional factor $c$ is not exploited in the sequel.
\end{remark}
	
The only difference between $(\partial_t - \Delta)U$ in the case of \eqref{e:TPsi} compared to 
that of \eqref{e:Phi} is that we obtain an additional term 
$\beta F_n'(\langle v + \<1>, \psi \rangle) \psi$, which is easily bounded by
	\begin{equs}
	\beta\|F_n'(\langle v + \<1>, \psi \rangle) \psi \|_{D} &\leq \beta\|\langle v + \<1>, \psi \rangle^3 \psi\|_{D} 
	 \lesssim \sum_{j = 0}^3 \|\langle v, \psi \rangle^j \|_{P_R} \|\langle \<1>, \psi\rangle^{3-j}\|_{D} \\
	 &\lesssim \|v\|_D^3 + [\<1>]_{\kappa}^3
	 \lesssim T^{-2} \|v\|_D \;,
	\end{equs}
	with constants uniform in $n$. Here, we made use of the bounds $T \le 1 \wedge \|v\|_D^{-1}$ and  $[\<1>]_\kappa \lesssim \|v\|_D^{1/2-\kappa} \le \|v\|_D$, where the last inequality follows
	from the fact that \eqref{e:contrav} is assumed to fail. In particular, this merely contributes in
	an increase of the proportionality constant appearing in \eqref{e:wantedBoundU}.
	
The other step where the precise form of the equation matters is the bound
\cite[Eq.~4.28]{MW18a}, where the maximum appearing on the right-hand side should include
an additional term $\big\|\beta \big(F_n'(\langle v + \<1>, \psi \rangle) \psi\big)_T\big\|_{D'}^{1/3}$ coming from 
our additional nonlinearity (here the domain $D'$ is equal to $P_{r+R'}$ in their notation). 
Similarly to above, this is bounded by some multiple of $\beta^{1/3} \|v\|_{D'}$.
By choosing $\beta$ to be sufficiently small we can guarantee that this is bounded by 
$\frac12\|v\|_{D'}$, so that the bound \cite[Eq.~4.33]{MW18a} does indeed still hold.
This allows to show that the required bound $\|v\|_{D} \lesssim R^{-1}$ holds in the 
same way as \cite{MW18a}. 
\end{proof}

In Section~\ref{s: Extensions}, we also required a version of Theorem~\ref{t:continuum bound} that incorporated localisation in space. Whilst this result is closer in flavour to that of \cite{MW18a}, it poses one additional difficulty not present above. Since the non-linearity we consider is not local in space, one cannot hope to have control on the behaviour of $v$ without information from the entirety of the support of $\psi$ so that complete localisation in space is not possible. However, despite this barrier, a sufficiently strong result for our purposes is available.

\begin{theorem}\label{t: localised continuum bound}
		Fix a smooth, compactly supported function $\psi: \mathbb{R}^3 \to \mathbb{R}$ such that $\|\psi\|_\infty, \|D\psi\|_\infty \leq 1$ and fix also $\beta > 0$ and $\kappa > 0$ small enough. Additionally define $Q_R = (R^2, 1) \times [-N + R, N-R]^3$ for $N$ chosen so that $\operatorname{supp} \psi \subseteq \subseteq [-N,N]^3$. Then if $u$ solves \eqref{e:smooth Psi renorm} on $Q_0$ for this $\psi$ and $v = u - \<1>$, then for all $R$ satisfying $0 < R < c_0$ for a suitable sufficiently small $c_0$ one has the bound
		\begin{equ}[e:boundvfinal]
			\|v\|_{Q_R} \le C\max\Big\{ R^{-1},[\tau]_{\kappa, N}^\frac{2}{n_\tau (1-\kappa)}; \tau \in \fT\Big\}\;,
		\end{equ}
		where $n_\tau$ is the number of leaves appearing in $\tau$ and for $\tau \in \fT$, $[\tau]_{\kappa, N}$ is defined analogously to $[\tau]_\kappa$ but restricting spatial suprema to $[-N,N]^3$. Here $C, c_0$ are constants that are independent of $n$ and $\eps$, and $\|v\|_{Q_R}$ denotes the supremum norm over $Q_R$.
\end{theorem}
\begin{proof}
	As in the proof above, this follows with only minor modifications to the proof of \cite[Thm~2.1]{MW18a}. In fact, the only additional modification necessary to those already given in the proof of Theorem~\ref{t:continuum bound} is to restrict consideration to $R$ sufficiently small so that $\operatorname{supp} \psi \subseteq Q_R$ and to add as an additional termination condition in the recursive step of their argument in their Section 4.6 the condition that (in the notation of that section), $R_n$ is such that $\operatorname{supp}\psi \not \subseteq Q_{R_n}$. 
\end{proof}
\bibliographystyle{Martin}
\bibliography{bib}

\appendix
\section{Convergence of Discretisations of the Initial Condition}

This appendix is an addition to the version of this paper published in the Journal of Statistical Physics which includes the calculation necessary to obtain convergence of our choice of discretisation of the initial condition. We include it in the hope that it may be useful to those aiming to use similar tools.

Our aim is to show that for $\zeta \in C^{\alpha'}$, $\iota^\varepsilon \FP_\varepsilon \zeta \to \zeta$ in a suitable sense. We will write $\zeta^\varepsilon = \FP_\varepsilon \zeta$.

We note that 
$$\scal{\zeta - \iota^\varepsilon \zeta^\varepsilon, \psi_x^\lambda} = \zeta \left ( \sum_{y \in \Lambda_\varepsilon^3} \varepsilon^{-3} 1_{\square_\varepsilon^y}(\cdot) \int_{\square_\varepsilon^y} \psi_x^\lambda(\cdot) - \psi_x^\lambda(z) dz \right ).$$

Hence, it is natural to consider $\| \Psi_\varepsilon^\lambda \|_{\CB_{1,1}^{\frac12 + \bar{\kappa}}}$ where $$\Psi_\varepsilon^\lambda \defeq \sum_{y \in \Lambda_\varepsilon^3} \varepsilon^{-3} 1_{\square_\varepsilon^y}(\cdot) \int_{\square_\varepsilon^y} \psi_x^\lambda(\cdot) - \phi_x^\lambda(z) dz.$$

We have the following result. 
\begin{lemma}
	Let $\bar{\kappa} < \frac12 \kappa$ and let $\varepsilon \le \lambda$. Then $\| \Psi_\varepsilon^\lambda \|_{\CB_{1,1}^{\frac12 + \bar{\kappa}}} \lesssim \lambda^{-\frac12 - \kappa} \varepsilon^{\kappa - 2 \bar{\kappa}}$.
\end{lemma}

\begin{proof}
	First we consider the regime $0 < \delta < \varepsilon < \lambda$ which is the least straightforward.
	
	We write 
	\begin{equs}
		\scal{\Psi_\varepsilon^\lambda, \eta_u^\delta} & = \sum_{\substack{y \in \Lambda_\varepsilon^3 \\ |y-x| \lesssim \lambda}} \varepsilon^{-3} \iint 1_{\square_\varepsilon^y \times \square_\varepsilon^y}(v,z) [\psi_x^\lambda(v) - \psi_x^\lambda(z)] \eta_u^\delta (v) dv dz
	\end{equs}
	and estimate each term in the sum separately (but uniformly over choices of $\eta \in \CB_r^0$).
	
	If $B(u, \delta) \subseteq \square_\varepsilon^y$ then since $\eta$ annihilates constants, we can write
	\begin{equs}
		\varepsilon^{-3} \iint 1_{\square_\varepsilon^y \times \square_\varepsilon^y}(v,z) [\psi_x^\lambda(v) - & \psi_x^\lambda(z)]  \eta_u^\delta (v) dv dz \\ & = \int 1_{\square_\varepsilon^y}(v) [\psi_x^\lambda(v) - \psi_x^\lambda(u)] \eta_u^\delta (v) dv
	\end{equs}
	and utilise a straightforward H\"older estimate on $\psi_x^\lambda$ and the fact that $\eta_u^\delta$ has uniformly bounded $L^1$-norm to obtain a bound of order $\lambda^{-4} \delta$ for $u$ in this region.
	
	Now if the integral is non-zero and $u$ is not in the region considered above then $\varepsilon - \delta \lesssim |u-y| \lesssim \varepsilon + \delta$. In this region, applying the brutal H\"older estimate to $\psi_x^\lambda$ yields a bound of order $\lambda^{-4} \varepsilon$. 
	
	Therefore 
	\begin{equs}
		 \int_0^\varepsilon \int &\delta^{-\frac32 - \bar{\kappa}} \sup_{\eta \in \CB_r^0}  |\scal{\Phi_\varepsilon^\lambda, \eta_u^\delta}| du d\delta
		\\
		& \lesssim \int_0^\varepsilon \delta^{-\frac32 - \bar{\kappa}} \sum_{\substack{y \in \Lambda_\varepsilon^3 \\ |y-x| \lesssim \lambda}} \int \left ( \lambda^{-4} \delta 1_{u \in B(y, c\varepsilon)} + \lambda^{-4} \varepsilon 1_{\varepsilon - \delta \lesssim |u-y| \lesssim \varepsilon + \delta} \right ) du d\delta 
		\\
		& \lesssim \int_0^\varepsilon \delta^{-\frac12 - \bar{\kappa}} \lambda^{-1} d \delta
		 \lesssim \varepsilon^{\frac12 - \bar{\kappa}} \lambda^{-1}
		 \lesssim \varepsilon^{\kappa - \bar{\kappa}} \lambda^{-\frac12 - \kappa}\;,
	\end{equs}
	which is the required bound.
	
	The remaining regimes are simpler. In the regime $\varepsilon < \delta < \lambda$, applying the brutal H\"older bound to the term in $\psi_x^\lambda$ and using the uniform control on the $L^1$ of $\eta_u^\delta$ immediately yields 
	\begin{equs}
		|\scal{\Phi_\varepsilon^\lambda, \eta_u^\delta}| \lesssim \begin{cases}
			\varepsilon \lambda^{-4} \text{ if } |u - x| \lesssim \lambda,
			\\
			0 \text{ otherwise,}
		\end{cases}
	\end{equs}
	so that 
	\begin{equs}
		\int_\varepsilon^\lambda \int \delta^{-\frac32 - \bar{\kappa}}  \sup_{\eta \in \CB_r^0}  |\scal{\Phi_\varepsilon^\lambda, \eta_u^\delta}| du d\delta & \lesssim \int_\varepsilon^\lambda \varepsilon \lambda^{-1} \delta^{- \frac 32 - \bar{\kappa}} d \delta
		\\
		& \lesssim \lambda^{-\frac12 - \kappa} \varepsilon^{\kappa - \bar{\kappa}}.
	\end{equs}
	
	Finally, in the regime where $\varepsilon < \lambda < \delta$, we estimate $|\scal{\Psi_\varepsilon^\lambda, \eta_u^\delta}|$ by applying the brutal H\"older estimate to the term in $\psi_x^\lambda$, using $|\eta_u^\delta(v)| \lesssim \delta^{-3}$ and then using the straightforward $L^1$ control on the indicator function to get a bound of order $\lambda^{-1} \varepsilon \delta^{-3}$. 
	
	Since $\scal{\Psi_\varepsilon^\lambda, \eta_u^\delta} = 0$ for $\delta \lesssim |u - x|$, this implies that
	\begin{equs}
		\int_\lambda^1 \int \delta^{-\frac32 - \bar{\kappa}}  \sup_{\eta \in \CB_r^0}  |\scal{\Phi_\varepsilon^\lambda, \eta_u^\delta}| du d\delta & \lesssim \int_\lambda^1 \varepsilon \lambda^{-1} \delta^{-\frac32 - \bar{\kappa}} d \delta
		\\ 
		& \lesssim \lambda^{-\frac12 - \kappa} \varepsilon^{\kappa - 2 \bar{\kappa}}.
	\end{equs}
Collecting all these bounds implies the claim.
\end{proof}
Combining this lemma with standard duality results for Besov spaces then yields the following.
\begin{corollary}\label{c:init discretisation}
	$$\| \zeta ; \zeta^\varepsilon \|_{- \frac12 - \kappa}^{(\varepsilon)} \defeq \sup_{ x \in \DTorus} \sup_{\lambda \in (\varepsilon, 1]} \sup_{\psi \in \CB_r^0} \lambda^{\frac12 + \kappa} | \scal{\zeta - \iota^\varepsilon \zeta^\varepsilon, \psi_x^\lambda}| \to 0$$ as $\varepsilon \to 0$. 
\end{corollary}
\end{document}